\documentclass[12pt,a4]{article}
\usepackage[colorlinks=true, urlcolor=blue, linkcolor=blue, citecolor=magenta, linktocpage=true]{hyperref}
\usepackage{amsmath, amsfonts, amsthm, mathrsfs, mathtools, amssymb}
\theoremstyle{definition}
\newtheorem{dfn}{Definition}[section]
\newtheorem{rem}[dfn]{Remark}
\newtheorem{ex}[dfn]{Example}
\theoremstyle{plain}
\newtheorem{prop}[dfn]{Proposition}
\newtheorem{lem}[dfn]{Lemma}
\newtheorem{thm}[dfn]{Theorem}
\newtheorem{cor}[dfn]{Corollary}
\newtheorem{cond}[dfn]{Condition}

\numberwithin{equation}{section}

\usepackage[style=alphabetic,sorting=anyt,url=true]{biblatex}
\usepackage{bm}
\usepackage{comment}
\usepackage{xurl}
\usepackage{tocloft}

\newcommand{\ms}[1]{\mathscr{#1}}

\newcommand{\mb}[1]{\mathbb{#1}}
\newcommand{\localslope}{|\partial^- f|}
\newcommand{\localslopeu}{\localslope \circ u}
\newcommand{\cmsp}{\text{CMS}_p([0, a); u_0)}
\newcommand{\pcms}[2]{\text{CMS}_{#1}([0, {#2}); u_0)}
\newcommand{\cmspinfty}{\text{CMS}_p([0, +\infty); u_0)}
\newcommand{\params}{\bm{s}}
\newcommand{\paramt}{\bm{t}}
\newcommand{\tu}{\tilde{u}}

\newcommand{\ta}[1]{T^\ast_{#1}}
\newcommand{\sap}[2]{S_{{#1}\to{#2}}^\ast}
\newcommand{\sa}{S^\ast}
\newcommand{\ut}{\tilde{u}}
\newcommand{\cldomain}{\overline{D(f)}}
\newcommand{\evilam}{\text{EVI}_\lambda}
\newcommand{\ubar}{\bar{u}}

\usepackage{enumitem}

\setlist[enumerate]{topsep=0pt,itemsep=-1ex,partopsep=1ex,parsep=1ex}

\title{Transformation of $p$-gradient flows to $p'$-gradient flows in metric spaces}
\author{Sho Shimoyama}
\date{}

\begin{document}
\maketitle
\begin{abstract}
    We explicitly construct parameter transformations between gradient flows in metric spaces, called curves of maximal slope, having different exponents when the associated function satisfies a suitable convexity condition.
    These transformations induce the uniqueness of gradient flows for all exponents under a natural assumption which is satisfied in many examples.
    We also prove the regularizing effects of gradient flows.
    To establish these results, we directly deal with gradient flows instead of using variational discrete approximations which are often used in the study of gradient flows.
\end{abstract}

\setlength{\cftbeforesecskip}{0pt}
\tableofcontents

\section{Introduction}
The main goal of this study is to explicitly construct parameter transformations from $p$-gradient flows to $p'$-ones in complete metric spaces, for all $p, p' \in (1, +\infty)$.

For a function $f$ on a complete metric space, a $p$-gradient flow, also called \textit{$p$-curve of maximal slope}, is a continuous map $u$ from an interval $[0, a)$ to the space, and is a solution of the evolution equation in the form
\begin{align}
    \label{eq:intro_p_cms}
    (f \circ u)'(t) \leq - \frac{1}{p}|u'|^p(t) - \frac{1}{q} \localslope^q \circ u(t),
\end{align}
where $q$ is the conjugate exponent of $p$, and $|u'|$ and $\localslope$ are the speed of $u$ and the norm of the gradient of $f$ in a metric sense, respectively.
The evolution equation of the form~\eqref{eq:intro_p_cms} arises in several contexts: doubly nonlinear evolution equations on Wasserstein spaces~\cite{Otto1996DoublyDD,Agueh2005,caillet2024doubly}, on Banach spaces~\cite{colli1992,rosmiesav08,Hynd2016ADN}, and application to Rayleigh quotient~\cite{HYND20174873}. 

The question of what properties are consistent and what properties are different between $p$ and $p'$-gradient flows is mostly unclear.
Specifically, our attention is directed towards the uniqueness of $p$-gradient flows.

Muratori and Savar\'e in~\cite{muratori2020108347} showed that a $2$-gradient flow starting from the same point exists uniquely, for suitable convex, called \textit{semi-convex}, functions on many interesting spaces: e.g., $2$-Wasserstein spaces over separable Hilbert spaces; 
connected and complete smooth Riemannian manifolds; complete CAT$(k)$ spaces with $k \in \mb{R}$ which are metric spaces with curvature bounded from above~\cite{burago2001course}; 
locally compact $\text{RCD}(K, \infty)$ spaces with $K \in \mb{R}$ which are metric measure spaces with Ricci curvature bounded from below~\cite{Ambrosio_2014}. 
Furthermore, for such functions and metric spaces, they demonstrated that $2$-gradient flows satisfy regularizing effects including the Lipschitz continuity of the flows.
However, within the same settings, the unique existence and regularizing effects of $p$-gradient flows for $p\neq 2$ have been unrevealed, hindered by the nonlinearity associated with $p \neq 2$.
We also refer to \cite{KELL20162045} in which the uniqueness of $p$-gradient flows is proved for a specific convex function on $p$-Wasserstein spaces when $p \in (1, 2]$.

\subsection*{Results}
Let $(\ms{S}, d)$ be a complete metric space, and $f \colon \ms{S} \to (-\infty, +\infty]$ be lower semicontinuous and proper, i.e., $D(f)\coloneqq \{v \in \ms{S} \mid f(v) < +\infty\} \neq \emptyset$, and let $u_0 \in D(f)$.
We also define $p$-curves of maximal slope as follows:
\begin{dfn}[$p$-curves of maximal slope]
    \label{def:p_cms_in_intro}
    Let $p \in (1,+\infty)$.
    We say that $u \in AC_{loc}([0,a);\ms{S})$ is a \textit{$p$-curve of maximal slope} for $f$ starting from $u_0$ if we have $u(0) = u_0$, $f \circ u$ is non-increasing, and $u$ satisfies
    \begin{align}
        (f \circ u)'(t) \leq -\frac{1}{p} |u'|^p(t) - \frac{1}{q} \localslope^q(u(t)), \quad \ms{L}\text{-a.e. in } t \in [0, a),
    \end{align}
    where $|u'|(t) \coloneqq \lim_{s\to t}\frac{d(u(t), u(s))}{|t - s|}$, and $\localslope (v) \coloneqq \limsup_{w \to v}\frac{\left(f(v) - f(w)\right)^+}{d(v, w)}$ if $v\in D(f)$ and $+\infty$ otherwise.
    We denote by $\pcms{p}{a}$ the set of all $p$-curves of maximal slope for $f$ starting from $u_0$ defined on $[0, a)$.
\end{dfn}

The main goal of our study is the construction of parameter transformations from $p$-curves of maximal slope to $p'$-curves of maximal slope when $f$ satisfies the following convexity:
\begin{dfn}[$(p, \lambda)$-convexity]
    \label{def:p_lambda_convex}
    Let $p \in (1, +\infty)$ and $\lambda \in \mb{R}$.
    We say that a function $f \colon \ms{S} \to (-\infty, +\infty]$ is $(p, \lambda)$-convex if there exists a function $\psi \colon [0, 1] \to [0, 1]$ with $\lim_{t\:\downarrow\:0} \psi(t) = 0$ and $\psi(t) < 1$ at some $t \in (0, 1)$, and for any $v_0, v_1 \in \ms{S}$ there exists a curve $\gamma \colon [0, 1] \to \ms{S}$ from $v_0$ to $v_1$ such that
    \begin{align*}
        d(\gamma_t, \gamma_0) &\leq t d(\gamma_1, \gamma_0), \\
        f(\gamma_t) &\leq (1-t)f(\gamma_0) + tf(\gamma_1) - \lambda t(1-\psi(t))d^p(\gamma_0, \gamma_1) \text{ for every } t \in [0, 1].
    \end{align*}
\end{dfn}
Note that $(p, \lambda)$-convex functions are also $(p, \lambda')$-convex for any $\lambda' \leq \lambda$.
Therefore it is sufficient to consider the case $\lambda \leq 0$ in our results, excluding regularizing effects.

We are now in a position to introduce our main results.
In what follows, we denote by $\ta{u}$ a quantity depending on a $p$-curve of maximal slope $u$ and denote by $\sa_{p\to p'}$ a quantity depending on $p, p' \in (1, +\infty)$ and $u$.
\begin{thm}[Parameter transformation; the case $\lambda = 0$]
    \label{thm:p_cms_to_p_dash_cms_lambda_gt_0}
    Suppose that $f$ is $(p_0, \lambda)$-convex for some $(p_0, \lambda)$ satisfying $\lambda=0$ and $p_0 \in (1, +\infty)$. 
    Let $p\in(1, +\infty)$ and $u \in \pcms{p}{+\infty}$.
    Then for any exponent $p' \in (1, +\infty)$ there exists a non-increasing continuous map $\paramt_{p \to p'} \colon [0, +\infty) \to [0, \ta{u}]$ with
    \begin{enumerate}[nolistsep]
        \item the curve $u_{p'} \coloneqq u \circ \paramt_{p \to p'}$ belongs to $\text{CMS}_{p'}([0, +\infty); u_0)$;
        \item if $v \in \cmspinfty$ is equal to $u$, then we have $\paramt_{p \to p', u} = \paramt_{p \to p', v}$;
        \item the map $\paramt_{p' \to p}$ for $u_{p'}$ satisfies the inverse transformation $u = u_{p'} \circ \paramt_{p' \to p}$.
    \end{enumerate}
\end{thm}

\begin{thm}[Parameter transformation; the case $\lambda < 0$]
    \label{thm:p_cms_to_p_dash_cms_lambda_st_0}
    Suppose that $f$ is $(p_0, \lambda)$-convex for some $(p_0, \lambda)$ satisfying $\lambda < 0$ and $p_0 \in [2, +\infty)$. 
    Let $p \in (1, p_0]$ and $u \in \pcms{p}{+\infty}$.
    Then for any exponent $p' \in (1, p_0]$ if we have one of the following:
    \begin{enumerate}[label=(\alph*), nolistsep]
        \item $\sap{p}{p'} = +\infty$;
        \item $\sap{p}{p'} < +\infty$ and $\ta{u} < + \infty$;
        \item $\sap{p}{p'} < +\infty$, $\ta{u} = +\infty$ and there exists $\lim_{t \:\uparrow\: \ta{u}}u(t)$;
    \end{enumerate} 
    then there exists a non-increasing continuous map $\paramt_{p \to p'} \colon [0, +\infty) \to [0, \ta{u}]$ with
    \begin{enumerate}[nolistsep]
        \item the curve $u_{p'} \coloneqq u \circ \paramt_{p \to p'}$ belongs to $\text{CMS}_{p'}([0, +\infty); u_0)$;
        \item if $v \in \cmspinfty$ is equal to $u$, then we have $\paramt_{p \to p', u} = \paramt_{p \to p', v}$;
        \item the map $\paramt_{p' \to p}$ for $u_{p'}$ satisfies the inverse transformation $u = u_{p'} \circ \paramt_{p' \to p}$.
    \end{enumerate}
\end{thm}

Note that the condition $\lim_{t\:\uparrow\:\ta{u}} u(t)$ in (c) of Theorem~\ref{thm:p_cms_to_p_dash_cms_lambda_st_0} cannot be dropped.
We will give an example supporting this claim in Example~\ref{ex:assumption_cannot_be_dropped}.

These parameter transformations induce the uniqueness of a $p$-curve of maximal slope.
\begin{cor}[Uniqueness of $p$-curves of maximal slope; the case $\lambda = 0$]
    \label{cor:uniqueness_of_p_cms_lambda_gt_0_intro}
    Suppose that $f$ is $(p_0, \lambda)$-convex for some $(p_0, \lambda)$ satisfying $\lambda = 0$ and $p_0 \in (1, +\infty)$.
    For all $p, p' \in (1, + \infty)$ there exists a canonical bijective map $\Lambda \colon \pcms{p}{+\infty} \to \pcms{p'}{+\infty}$.
    
    In particular, if $\# \pcms{p}{+\infty} = 1$ for an exponent $p \in (1, +\infty)$, then $\# \pcms{p'}{+\infty} = 1$ for all $p' \in (1, +\infty)$.
\end{cor}

\begin{cor}[Uniqueness of $p$-curves of maximal slope; the case $\lambda < 0$]
    \label{cor:uniqueness_of_p_cms_lambda_st_0_intro}
    Suppose that $f$ is $(p_0, \lambda)$-convex for some $(p_0, \lambda)$ satisfying $\lambda < 0$ and $p_0 \in [2, +\infty)$.
    If there exists an exponent $p \in (1, p_0]$ such that
    \begin{align}
        \label{ass:intro_uniqueness_p_cms_for_all_a}
        \#\pcms{p}{a}=1 \quad \text{for any } a \in (0, +\infty],
    \end{align}
    then for any $p' \in (1, p_0]$ there exists $S^\ast \in [0, +\infty]$
    such that
        $S^\ast =  \sup \{a \in (0, +\infty] \mid \# \pcms{p'}{a} \geq 1\}$
    and there exists a canonical bijective map $\Lambda \colon \pcms{p}{+\infty} \to \pcms{p'}{S^\ast}$.
\end{cor}

We give two comments for Corollary~\ref{cor:uniqueness_of_p_cms_lambda_st_0_intro}.
For the one, it is previously known in~\cite[Theorem~4.2]{muratori2020108347} that the assumption~\eqref{ass:intro_uniqueness_p_cms_for_all_a} is satisfied among many concrete spaces and $(2, \lambda)$-convex functions on them. See Section~\ref{subsec:applications} for more details.
For another, the condition $p' \in (1, p_0]$ is essential.
We will give an example such that the $p$-curve of maximal slope is unique for $p\in (1, p_0]$, however, $p$-curves of maximal slope are not unique for $p \in (p_0 , +\infty)$. See Example~\ref{ex:p_cms_are_not_unique_for_p_gt_p_0}.

To prove Theorems~\ref{thm:p_cms_to_p_dash_cms_lambda_gt_0} and \ref{thm:p_cms_to_p_dash_cms_lambda_st_0} we show a local convexity of $f$ along $p$-curves of maximal slope.
This convexity also induces regularizing effects:
\begin{thm}[Regularizing effects]
    \label{thm:regularizing_effects}
    Let one of the following hold:
    \begin{enumerate}[label=(\alph*), nolistsep]
        \item $\lambda\geq0$, $p_0 \in (1, +\infty)$ and $p \in (1, +\infty)$;
        \item $\lambda < 0$, $p_0 \in [2, +\infty)$ and $p \in (1, p_0]$.
    \end{enumerate}
    Suppose that $f$ is $(p_0, \lambda)$-convex.
    Then every $u \in \pcms{p}{a}$ satisfies the following properties:
    \begin{enumerate}[leftmargin=*, nolistsep]
        \item the curve $t \mapsto u(t)$ and the map $t \mapsto f\circ u(t)$ are locally Lipschitz on $(0, a)$, and $u(t) \in D(\localslope)$ for every $t \in (0, a)$;
        \item the right limits
        \begin{align*}
            (f\circ u)'_+(t) \coloneqq \lim_{h\:\downarrow\:0} \frac{f(u(t+h)) - f(u(t))}{h}, |u'_+|(t) \coloneqq \lim_{h\:\downarrow\:0}\frac{d(u(t+h), u(t))}{h}
        \end{align*}
        are exist for every $t \in (0, a)$, satisfy that for every $t \in (0, +\infty)$
        \begin{align}
            \label{eq:regularizing_right_energy_identity}
            (f\circ u)'_+(t) = - \localslopeu(t)|u'_+|(t) = - |u'_+|^p(t) = - \localslope^q\circ u(t),
        \end{align}
        and these are almost right continuous in the sense that there exists $A \subset [0, a)$ with full $\ms{L}$-measure in $[0, a)$ and 
        \begin{align}
            \label{eq:regularizing_right_approximate_cont}
            \lim_{A \ni t' \:\downarrow\: t} \localslopeu(t') = \localslopeu(t) \quad \text{for every } t \in (0, a).
        \end{align}
        In addition, if $u_0 \in D(\localslope)$, then  \eqref{eq:regularizing_right_energy_identity} and \eqref{eq:regularizing_right_approximate_cont} hold at $t=0$.
        Moreover, if $\lambda \geq 0$, then $\localslopeu$ is non-increasing and right continuous on $[0, a)$;

        \item if $\lambda \geq 0$, then we have for all $t_0, t \in [0, a)$ with $t_0 < t$
        \begin{align*}
            \localslope^q(u(t)) \leq \frac{f(u(t_0)) - \inf_{\ms{S}}f}{t - t_0}, \quad
             \frac{1}{q}\localslope^q (u(t)) \leq \frac{f(u(t_0)) - f_{t-t_0}(u(t_0))}{t-t_0},
        \end{align*}
        where $f_t(v) \coloneqq \inf_{w \in \ms{S}}\left\{f(w) + \frac{1}{p t^{p-1}}d^p(v, w) \right\}$ and $q$ is the conjugate exponent of $p$:
        
        \item if $\lambda > 0$, then denoting by $\ubar$ the unique minimizer of $f$ it holds that for any $t \in (0, a)$
        \begin{align*}
                \lambda d^{p_0}(u(t), \ubar)
                \leq f(u(t)) - f(\ubar)
                \leq \frac{1}{q_0 \lambda^{q_0/p_0}} \localslope^{q_0}(u(t)),
        \end{align*}
        where $q_0$ is the conjugate exponent of $p_0$.
        \end{enumerate}
\end{thm}
Note that the statement (iv) is already known in~\cite{ambrosioGradientFlowsMetric2008,ohta2023gradient} however we quote it for convenience.
On the other hand, the other statements have been known only in partial cases, such as 
for a $p$-curves of maximal slope which is obtained by variational discrete approximations called \textit{generalized minimizing movements} in~\cite[Theorem~2.4.15]{ambrosioGradientFlowsMetric2008}.

\subsection*{Plan of proofs}
The key facts to prove Theorems~\ref{thm:p_cms_to_p_dash_cms_lambda_gt_0} and \ref{thm:p_cms_to_p_dash_cms_lambda_st_0} are
\begin{enumerate}[nolistsep]
    \item (Corollary~\ref{cor:lower_bound_of_local_slope}) the existence of local positive lower bounds of the slope $\localslope$ along $p$-curves of maximal slope;
    \item (Lemma~\ref{lem:upper_bound_of_local_slope}) the existence of local finite upper bounds of $\localslope$ along $p$-curves of maximal slope.
\end{enumerate}

To show the first point we prove a new fact that the slope $\localslope$ along a $p$-curve of maximal slope remains zero and the curve is stationary once the slope is zero.
See Proposition~\ref{prop:positivity_of_localslope}.

The second point follows from another new fact that $f$ is locally convex along the arclength-parametrization of a $p$-curve of maximal slope, which we proved in Lemma~\ref{lem:convexity_of_f_circ_u_tilde}.
This convexity is derived via a proof by contradiction.
Since this plays a key role, let us give a sketch of the proof.

In order to simplify our arguments, consider the case $\lambda \geq 0$ (i.e., $f$ is convex) and the original parametrization of a $p$-curve of maximal slope $u$.
If $f \circ u$ is not convex, then there exist three points $s_0 < \hat{s} < s_1$ such that $f \circ u(\hat{s})$ is greater than a convex combination $\varphi$ of $f\circ u(s_0)$ and $f \circ u(s_1)$.
Then there exists a point $s_l$ such that $(f \circ u)'(s_l)$ is greater than the slope of $\varphi$.
We take a point $s_r > s_l$ such that the slope of $\psi$, which is a convex combination of $f\circ u(s_l)$ and $f\circ u(s_r)$, is equal to the slope of $\varphi$.
Since $f$ is convex, there exists a curve $\gamma$ from $u(s_l)$ to $u(s_r)$ differ from $u$ such that $f \circ \gamma$ is less than or equal to $\psi$.
Recalling that $(f \circ u)'(s_l)$ is greater than the slope of $\psi$, the value of $f$ decreases more along $\gamma$ than along $u$.
This is a contradiction to the fact that $u$ is a gradient flow, i.e., it represents the steepest descent direction.

Combining the two points (i) and (ii), we establish Theorems~\ref{thm:p_cms_to_p_dash_cms_lambda_gt_0} and \ref{thm:p_cms_to_p_dash_cms_lambda_st_0} by constructing explicit parameter transformations. 
Moreover the local convexity of $f$ proved in Lemma~\ref{lem:convexity_of_f_circ_u_tilde} allows us to utilize the similar techniques used to derive regularizing effects in~\cite{ambrosioGradientFlowsMetric2008,muratori2020108347}, 
realizing Theorem~\ref{thm:regularizing_effects}.

The uniqueness result Corollary~\ref{cor:uniqueness_of_p_cms_lambda_gt_0_intro} (resp. Corollary~\ref{cor:uniqueness_of_p_cms_lambda_st_0_intro}) is a direct consequence of Theorem~\ref{thm:p_cms_to_p_dash_cms_lambda_gt_0} (resp. Theorem~\ref{thm:p_cms_to_p_dash_cms_lambda_st_0}), which is proved in Corollary~\ref{cor:uniqueness_p_cms_lambda_gt_0} (resp. Corollary~\ref{cor:uniqueness_p_cms_lambda_st_0}).

\subsection*{Acknowledgement}
I would like to express my gratitude to Professor Shouhei Honda, my master's program supervisor, for his numerous valuable comments and support.

\section{Preliminaries}
\label{sec:preliminaries}
In this section, we recall some basic notions and results related to the gradient flows in metric spaces, called $p$-curves of maximal slope.
Let $(\ms{S}, d)$ be a complete metric space, and $f \colon \ms{S} \to (-\infty, +\infty]$ be lower semicontinuous and proper, i.e., $D(f) \coloneqq \{v \in \ms{S} \mid f(v) < +\infty\} \neq \emptyset$ throughout this section.
We represent an interval in $\mb{R}$ by $I$, and also call any continuous map from an interval to a metric space a \textit{curve}.
We refer to~\cite{ambrosioGradientFlowsMetric2008, ohta2023gradient} for the details.

Before starting the general setting, firstly, we consider a finite-dimensional Euclidian space $\ms{S} \coloneqq \mb{R}^n$ with a Euclidian distance $d$.
For a smooth function $f \colon \mb{R}^n \to \mb{R}$, we say that a smooth curve $u \colon I \to \mb{R}$ is a $p$-gradient flow for $p\in(1, +\infty)$ if it satisfies the following differential equation
\begin{align}
    \label{eq:grad_flow_1}
    \|u'(t)\|^{p-2} \cdot u'(t) = -\nabla f(u(t)) \quad \text{for each } t \in I.
\end{align}
Equation~\eqref{eq:grad_flow_1} cannot be defined on general metric spaces due to the lack of a linear structure.

On the other hand Young's inequality and chain rule reveals that \eqref{eq:grad_flow_1} holds if and only if the following holds
\begin{align}
    \label{eq:grad_flow_2}
    (f \circ u)'(t) \leq -\frac{1}{p}\|u'(t)\|^p - \frac{1}{q} \|\nabla f(u(t))\|^q \quad \text{for each } t \in I,
\end{align}
where $q$ is the conjugate exponent of $p$.
Note that in~\eqref{eq:grad_flow_2}, without linear structure of $\mb{R}^n$, we can define $\|u'(t)\|$ by
\begin{align*}
    |u'|(t) \coloneqq \lim_{s\to t}\frac{d(u(t), u(s))}{|t - s|} = \|u'(t)\|
\end{align*}
and $\|\nabla f(v)\|$ by 
\begin{align*}
    \localslope(v) \coloneqq \limsup_{\mb{R} \ni w \to v} \frac{\left(f(v) - f(w)\right)^+}{d(v, w)} = \|\nabla f(v)\|,
\end{align*}
where $a^+ \coloneqq \max\{a, 0\}$.
This motivates the subsequent discussion.

\subsection{Absolutely continuous curve}
\begin{dfn}[Absolutely continuous curves]
  Let $p \in [1, +\infty]$.
  We say that $v \colon I \rightarrow \ms{S}$ is a $p$-absolutely (resp. $p$-locally absolutely) continuous curve if there exists $m \in L^p(I)$ (resp. $m \in L^p_{loc}(I))$ such that
  \begin{align}
    \label{eq:absolutely_continuous_curve}
  d(v(s), v(t)) \leq \int_s^t m(r) dr \quad \text{for any } s, t \in I \text{ with } s \leq t.
  \end{align}
  Denoted by $AC^p(I;\ms{S})$ (resp. $AC^p_{loc}(I; \ms{S})$) the set of all $p$-absolutely (resp. $p$-locally absolutely) continuous curves from $I$ to $\ms{S}$.
  In the case $p = 1$, we simply denote it with $AC(I;\ms{S})$ (resp. $AC_{loc}(I; \ms{S})$).
\end{dfn}

Any $p$-absolutely continuous curve is $\ms{L}\text{-a.e.}$ differentiable in the following sense.
See~\cite[Theorem~1.1.2] {ambrosioGradientFlowsMetric2008} for the proof.
\begin{thm}[Metric derivative]
  Let $p \in [1,+\infty]$.
  Then for any curve $v \in AC^p(I;\ms{S})$ (resp. $v \in AC^p_{loc}(I; \ms{S})$) the metric derivative
  \begin{align*}
    |v'|(t) \coloneqq \lim_{s \to t} \frac{d(v(s), v(t))}{|s - t|}    
  \end{align*}
  exists for $\ms{L}$-a.e. in $t \in I$.
  Moreover, the function $t \mapsto |v'|(t)$
  belongs to $L^p(I)$ (resp. $L^p_{loc}(I)$) and it is minimal in the sense:
  $
    |v'|(t) \leq m(t), \ms{L} \text{-a.e. } t \in I    
  $
  for every function $m$ satisfying~\eqref{eq:absolutely_continuous_curve}.
\end{thm}

\subsection{Local slope, \texorpdfstring{$p$}{p}-curves of maximal slope and upper gradient}
We define the \textit{local slope} $\localslope$ of $f$ by 
  \begin{align*}
  \localslope(v) \coloneqq \begin{cases}
    \limsup_{w \to v} \frac{(f(v) - f(w))^+}{d(v, w)}, & v \in D(f)\\
    +\infty, & v \notin D(f)
  \end{cases}.
  \end{align*}

Recalling \eqref{eq:grad_flow_2}, we provide the following definition:
\begin{dfn}[$p$-curves of maximal slope]
  \label{def:p_cms}
  Let $p \in (1,+\infty)$ and $u_0 \in D(f)$.
  We say that $u \in AC_{loc}([0,a);\ms{S})$ is a \textit{$p$-curve of maximal slope} for $f$ starting from $u_0$ if we have $u(0) = u_0$ and there exists a non-increasing map $\varphi \colon [0, a) \to \mb{R}$ such that it is $\ms{L}$-a.e. equal to $f \circ u$ and satisfies the following differential inequality:
  \begin{align}
  \label{eq:p_cms}
  \varphi'(t) \leq -\frac{1}{p} |u'|^p(t) - \frac{1}{q} \localslope^q(u(t)), \quad \ms{L}\text{-a.e. in } t \in [0, a).
  \end{align}
  We denote by $\cmsp$ the set of all $p$-curves of maximal slope for $f$ starting from $u_0$ defined on $[0, a)$.
\end{dfn}

Note that Definition~\ref{def:p_cms} and Definition~\ref{def:p_cms_in_intro} are the same in our study; see Propositions~\ref{prop:local_slope_is_a_strong_upper_gradient} and~\ref{prop:minimizer_lambda_gt_0} for more details.

When $f$ is a smooth function on a Euclidian space, a non-negative continuous function $g \colon \mb{R}^n \to [0, +\infty)$ satisfies $\|\nabla f\| \leq g$ if and only if it satisfies $|f(v(t)) - f(v(s))| \leq \int_{s}^{t} g(v(r)) \cdot |v'|(r) dr$ for all $v \in AC(I; \ms{S})$ and all $s, t \in I$ with $s \leq t$.
Such function $g$ is said to be an upper gradient.
\begin{dfn}[Strong upper gradient]
  We say that a function $ g : \ms{S} \to [0,+\infty] $ is a \textit{strong upper gradient} for $ f $ if for every $ v \in AC(I;\ms{S}) $ with $v(I) \subset D(f)$ the function $t \mapsto g \circ v(t) $ is Borel and 
  \begin{align}
    \label{eq:sug}
    |f(v(t)) - f(v(s))| \leq \int_s^t g\circ v(r)|v'|(r) dr \quad \text{for any } s, t \in I \text{ with } s \leq t.
  \end{align}
  In particular, if $ g \circ v |v'| \in L^1(I) $, then $ f \circ v $ is absolutely continuous and
  \begin{align*}
    |(f \circ v)'(t)| \leq g\circ v(t)|v'|(t), \quad \ms{L}\text{-a.e. in } t \in I.
  \end{align*}
\end{dfn}

\begin{rem}
    A weak notion of the strong upper gradients, called \textit{weak upper gradients}, are proposed in~\cite[Definition~1.2.2]{ambrosioGradientFlowsMetric2008}; however we do not use this weak notion in our study.
\end{rem}

When the local slope is a strong upper gradient, $p$-curves of maximal slope have nice properties as follows.
\begin{prop}
  \label{prop:p_cms_energy_identity}
  Let $p \in (1, +\infty)$ and the local slope $\localslope$ be a strong upper gradient.
  Then for any $u \in \cmsp$
  we have the following properties:
  \begin{enumerate}
      \item the non-increasing map $\varphi$ as in Definition~\ref{def:p_cms} is equal to $f\circ u$ and it is locally absolutely continuous on $[0, a)$;
      \item the maps $t \mapsto |u'|(t)$, $t \mapsto \localslope \circ u(t)$ belong to $L^p_{loc}([0, a))$, $L^q_{loc}([0, a))$, respectively;
      \item the energy dissipation equation holds: for any $s, t \in [0, a)$ with $s \leq t$
      \begin{align*}
        f \circ u(t) - f\circ u(s) = - \frac{1}{p}\int_{s}^{t} |u'|^p(r)dr - \frac{1}{q} \int_{s}^{t} \localslope^q \circ u(r)dr.
      \end{align*}
  \end{enumerate}
  In particular, we have the energy identity
  \begin{align}
    \begin{split}
    \label{eq:energy_identity}
    &(f\circ u)'(t) = - \localslopeu(t) \cdot |u'|(t)
    = -|u'|^p(t) = -\localslope^q \circ u(t),\\
    &\ms{L}\text{-a.e. in } t \in [0, a).
    \end{split}
  \end{align}
\end{prop}

\begin{proof}
  Since $u$ satisfies the inequality~\eqref{eq:p_cms} and $\varphi$ is non-increasing, we have $|u'| \in L_{loc}^p([0, a))$ and $\localslope \circ u \in L^q_{loc}([0, a))$.
  Therefore the map $t \mapsto \localslope\circ u(t) \cdot |u'|(t)$ belongs to $L_{loc}^1([0, a))$.
  This implies that $f \circ u$ is locally absolutely continuous on $[0, a)$ because the local slope $\localslope$ is a strong upper gradient for $f$.
  We get $f\circ u = \varphi$ because continuous and monotone functions that are equal in $\ms{L}\text{-a.e.}$ points are identical at every point.
  Combining the integral of \eqref{eq:p_cms} and \eqref{eq:sug}, we get the energy dissipation equation.
\end{proof}

\subsection{\texorpdfstring{$(p, \lambda)$}{plambda}-convex function}
In this section, we recall the properties of the local slope for $(p, \lambda)$-convex functions: see Definition~\ref{def:p_lambda_convex}.

\begin{rem}
    Our definition of $(p, \lambda)$-convexity generalizes $(\lambda, p)$-convexity in~\cite[Definition~2.5]{ROSSI2011205} and $(p, \lambda)$-convexity in~\cite[Definition~4.1]{ohta2023gradient}, which are defined using $\psi(t)=t$, $\psi(t) = t^{p-1}$, respectively.
\end{rem}

\begin{ex}
    \begin{enumerate}
        \item Any convex function on a Euclidian space, more generally any geodesically convex function on a Riemannian manifold or a Banach space, is $(p, \lambda)$-convex for any $\psi$ as in Definition~\ref{def:p_lambda_convex}, $p\in(1, +\infty)$ and $\lambda \leq 0$.
        \item Let $p\geq2$ and $(\ms{S}, d)$ be a $L^p$-space over a measure space, more generally a $p$-uniformly convex metric space.
        For any point $w \in \ms{S}$, the function $d^p(\cdot, w)$ is $(p, \lambda)$-convex for $\psi(t)=t$ and some $\lambda > 0$; see~\cite[(3.4)]{XU19911127} for $L^p$ spaces and see~\cite[Definition~3.2]{Naor_Silberman_2011} for $p$-uniformly convex metric spaces.
    \end{enumerate}
\end{ex}

Before introducing the properties of the local slope, we present some remarks on another convexity condition related to $(p,\lambda)$-convexity.
Namely,
\begin{cond}
    \label{cond:another_convexity}
    Let $p\in(1, +\infty)$ and $\lambda \in \mb{R}$.
    For any $v_0, v_1 \in \ms{S}$ there exists a curve $\gamma \colon [0, 1] \to \ms{S}$ such that
    \begin{align*}
        \begin{split}
        F_p(\gamma_t; \tau, v_0) \leq & (1-t)F_p(\gamma_0; \tau, v_0) + t F_p(\gamma_1; \tau, v_0)\\
        & - t(1-t^{p-1})\left(\frac{1}{p\tau^{p-1}} + \lambda\right) d^p(\gamma_1, \gamma_0)
        \end{split}
    \end{align*}
    for every $t \in [0, 1]$ and $\tau$ with $\tau^{p-1} \in \left(0, \frac{1}{\lambda^-} \right)$, where $F_p(w; \tau, v) \coloneqq f(w) + \frac{1}{p\tau^{p-1}}d^p(w, v)$, $\lambda^- \coloneqq \max\{-\lambda, 0\}$ and $\frac{1}{\lambda^-}\coloneqq +\infty$ if $\lambda^-$ = 0.
\end{cond}

\begin{rem}
    Our main results hold if we assume that Condition~\ref{cond:another_convexity} holds for some $(p, \lambda)$ instead of $(p, \lambda)$-convexity of $f$ 
    because, in the proof of our results, it is sufficient to use Propositions~\ref{prop:local_slope_is_a_strong_upper_gradient} and~\ref{prop:minimizer_lambda_gt_0} introduced below; and these propositions hold when Condition~\ref{cond:another_convexity} are satisfied.
\end{rem}

\begin{ex}
    \begin{enumerate}
        \item Any $(p, \lambda)$-convex function for any $(p, \lambda)$ and $\psi(t)=t^{p-1}$ satisfies Condition~\ref{cond:another_convexity} for $p$ and $\lambda$.
        \item For a 2-Wasserstein space over a separable Hilbert space and any $\lambda < 0$, it is unclear that any $\lambda$-convex function along \textit{generalized geodesics} on such a space is $(2, \lambda)$-convex;
        however such a function satisfies Condition~\ref{cond:another_convexity} for $p=2$ and $\lambda$.
        See \cite[Section~9.2]{ambrosioGradientFlowsMetric2008} for more details.
    \end{enumerate}
\end{ex}

We are now back in the position to deal with the local slope of $(p,\lambda)$-convex functions.
The slope is a strong upper gradient, which allows us to use the valuable result -Proposition~\ref{prop:p_cms_energy_identity}-.
\begin{prop}
    \label{prop:local_slope_is_a_strong_upper_gradient}
    Let $f$ be $(p, \lambda)$-convex for some $p\in(1, +\infty)$ and $\lambda \in \mb{R}$.
    Then the local slope $\localslope$ is a lower semicontinuous strong upper gradient for $f$.
    Moreover the following global formula holds for any $v \in D(f)$:
    \begin{align}
        \localslope (v) &= \sup_{w \neq v} \left\{\frac{f(v) - f(w)}{d(v,w)} + \lambda d^{p-1}(v, w) \right\}^+ \notag \\
        \label{eq:global_formula}
        &= \sup_{w \neq v} \left\{\frac{f(v) - f(w)}{d(v,w)} - \lambda^- d^{p-1}(v, w) \right\}^+.
    \end{align}
\end{prop}
\begin{proof}
    See~\cite[Theorem~2.4.9 and Corollary~2.4.10]{ambrosioGradientFlowsMetric2008} for $p=2$ or~\cite[Theorem~4.8 and Corollary~4.9]{ohta2023gradient} for general $p \in (1, +\infty)$.
\end{proof}

The next result is well-known.
The proofs can be seen in~\cite[Lemma~2.4.13]{ambrosioGradientFlowsMetric2008} for the case $p=2$ and in~\cite[Lemma~4.13]{ohta2023gradient} for the case $p \in (1, +\infty)$.
\begin{prop}
    \label{prop:minimizer_lambda_gt_0}
    Let $f$ be $(p, \lambda)$-convex for some $p\in(1, +\infty)$ and $\lambda \in \mb{R}$.
    If $\lambda > 0$, then we have
    \begin{align*}
        f(v) - \inf_{\ms{S}}f \leq \frac{1}{q \lambda^{q/p}}\localslope^{q}(v) \quad \text{for any } v \in D(f),
    \end{align*}
    where $q$ is the conjugate exponent of $p$.

    Moreover, if $D(\localslope) \neq \emptyset$ and $\lambda > 0$, then $f$ admits the unique minimizer $\bar{u}$ and we have
    \begin{align*}
        \lambda d^{p}(v, \bar{u}) \leq f(v) - f(\bar{u}) \quad \text{for any } v \in D(f).
    \end{align*}
\end{prop}

\section{Local slope's bounds along \texorpdfstring{$p$}{p}-curves of maximal slope}
\label{sec:bounds_of_local_slope}
The main purpose of this section is to show that the local slope is bounded by a positive finite constant from above and below locally along $p$-curves of maximal slope.
Throughout this section, let $(\ms{S},d)$ be a complete metric space, $f \colon \ms{S} \to (-\infty, +\infty]$ be a proper and lower semicontinuous, and $u_0 \in D(f)$.

\begin{prop}
    \label{prop:positivity_of_localslope}
    Suppose that $f$ is $(p_0, \lambda)$-convex for some $p_0 \in (1, +\infty)$ and $\lambda \leq 0$.
    Let $p \in (1, +\infty)$ if $\lambda = 0$ or $p \in (1, p_0]$ if $\lambda < 0$, and $u \in \pcms{p}{a}$.
    If there exists $t^\ast \in [0, a)$ such that $\localslopeu(t^\ast) = 0$,
    then we have
    \begin{enumerate}
        \item $\localslopeu(t) \equiv 0$ on $[t^\ast, a)$;
        \item $u(t) \equiv u(t^\ast)$ on $[t^\ast, a)$.
    \end{enumerate}
\end{prop}

\begin{proof}
    Note that the energy identity~\eqref{eq:energy_identity} holds and $\localslope$ is lower semicontinuous; (i) follows from (ii).
    Thus let us check (ii).
    Let $A \coloneqq \{t \in [t^\ast, a) \mid u(t) = u(t^\ast) \}$.
    We fix any $T \in (t^\ast, a)$, then consider the restriction $A' \coloneqq A \cap [t^\ast, T]$.
    It is enough to show that there exists $\delta > 0$ such that $[t, t+\delta] \cap [t^\ast, T] \subset A'$ for any $t \in A'$.
    The uniform continuity of $u$ on $[t^\ast, T]$ implies the existence of $\delta_0 > 0$ satisfying that $d(u(t), u(s)) \leq 1$ for any $t, s \in [t^\ast, T]$ with $|s - t| \leq \delta_0$.
    Let $\delta \coloneqq \min\{\delta_0, \left(\frac{1}{2\lambda^-}\right)^{q/p}\}$.
    Suppose that there exists $t_0 \in A'$ with $[t_0, t_0+\delta] \cap [t^\ast, T] \not\subset A'$; hence we can take $t_1 > t_0$ such that $t_1 - t_0 \leq \left(\frac{1}{2\lambda^-}\right)^{q/p}$ and $0 < d(u(t_0), u(t_1)) \leq 1$.
    The global formula~\eqref{eq:global_formula} of the local slope gives
    \begin{align}
        \label{eq:positivity_of_localslope_1}
        f(u(t_0)) - f(u(t_1)) \leq \lambda^- d^{p_0}(u(t_0), u(t_1)).
    \end{align}
    Combining~\eqref{eq:positivity_of_localslope_1} with the energy identity~\eqref{eq:energy_identity} we get
    \begin{align*}
        &d^p(u(t_0), u(t_1))
        \leq \int_{t_0}^{t_1} |u'|^p(r)dr \cdot (t_1 - t_0)^{p/q}\\
        &= \left\{ f(u(t_0)) - f(u(t_1)) \right\} \cdot (t_1 - t_0)^{p/q}
        \leq \begin{cases}
            0 & \text{if } \lambda = 0\\
            \frac{1}{2}d^{p_0}(u(t_0), u(t_1)) & \text{if } \lambda < 0.
        \end{cases}
    \end{align*}
    This is a contradiction.
    Thus we can get the desired result.
\end{proof}

Based on this proposition, we define the supremum time $\ta{u}$ representing the positivity of the local slope along a $p$-curve of maximal slope $u$ by,
\begin{align}
    \label{eq:t_ast_u}
    \ta{u} \coloneqq \sup \{t \in [0, a) \mid \localslopeu(t) > 0\}.
\end{align}
The following lemma immediately follows from Proposition~\ref{prop:positivity_of_localslope} and the lower semicontinuity of the local slope.
\begin{cor}[Slope's lower bound along with $p$-curves of maximal slope]
    \label{cor:lower_bound_of_local_slope}
    Suppose that $f$ is $(p_0, \lambda)$-convex for some $p_0 \in (1, +\infty)$ and $\lambda \leq 0$.
    Let $p \in (1, +\infty)$ if $\lambda = 0$ or $p \in (1, p_0]$ if $\lambda < 0$, and $u \in \pcms{p}{a}$.
    For every $T \in (0, \ta{u})$ there exists a positive constant $c>0$ such that 
    \begin{align*}
        \localslopeu(t) \geq c \quad \text{for every } t \in [0, T].
    \end{align*}
    Moreover, if we also have $\ta{u} < a$, then the following hold:
    \begin{enumerate}
        \item $\localslopeu(t) \equiv 0$ on $[\ta{u}, a)$;
        \item $ u(t) \equiv u(\ta{u})$ on $[\ta{u}, a)$.
    \end{enumerate}
\end{cor}

The following result is an enhanced version of a well-known reparametrization result, e.g., introduced in~\cite[Lemma~1.1.4]{ambrosioGradientFlowsMetric2008}.
\begin{lem}[Arc-length parameter]
    \label{lem:arc_length_param}
    Suppose that $f$ is $(p_0, \lambda)$-convex for some $p_0 \in (1, +\infty)$ and $\lambda \leq 0$.
    Let $p \in (1, +\infty)$ if $\lambda = 0$ or $p \in (1, p_0]$ if $\lambda < 0$.
    For each $u \in \cmsp$ there exist a strictly increasing locally absolutely continuous map $\params \colon [0, \ta{u}) \to [0, \sa_u)$ and its inverse map $\paramt \colon [0, \sa_u) \to [0, \ta{u})$ satisfying the following:
    \begin{enumerate}[nolistsep]
        \item it holds that $\params'(t) = |u'|(t) \neq 0$,  $\ms{L}$-a.e. in $t \in [0, \ta{u})$;
        \item the map $s \mapsto \paramt(s)$ is locally Lipschitz on $[0, \sa_u)$, and it holds that \begin{align}
            \label{eq: differential_of_paramt}
            \paramt'(s) = \frac{1}{|u'|(\paramt(s))},\quad \ms{L}\text{-a.e. in } s\in[0, \sa_u);
        \end{align}
        \item the curve $\tilde{u} \coloneqq u \circ \paramt$ is $1$-Lipschitz on $[0, \sa_u)$;
    \end{enumerate}
    where the constant $\sa_u$ is defined to be $\sa_u \coloneqq \int_{0}^{\ta{u}} |u'|(r)dr$.
    
    We call the map $\paramt$ the arc-length parameter of $u$.
\end{lem}

\begin{proof}
    We define $\params \colon [0, \ta{u}) \to [0, \sa_u)$ by $t \mapsto \int_{0}^{t} |u'|(r)dr$.
    Proposition~\ref{prop:p_cms_energy_identity} yields that $\params$ is locally absolutely continuous on $[0, \ta{u})$.
    For every $T \in (0, \ta{u})$ Corollary~\ref{cor:lower_bound_of_local_slope} and the energy identity~\eqref{eq:energy_identity}
    show that there exists a constant $c > 0$ satisfying
    \begin{align}
        \label{eq:arc_length_param_positivity_of_s}
        \params(t_2) - \params(t_1) \geq c (t_2 - t_1), \quad 0 \leq \forall t_1 \leq \forall t_2 \leq T;
    \end{align}
    hence $\params$ is strictly increasing.
    This completes the proof of (i).

    Let $\paramt$ be the inverse map of $\params$.
    For every $S \in (0, \sa_u)$ \eqref{eq:arc_length_param_positivity_of_s} induces the existence of $c > 0$ such that
    \begin{align*}
        \frac{\paramt(s_2) - \paramt(s_1)}{s_2 - s_1}
        = \frac{1}{\frac{\params(\paramt(s_2)) - \params(\paramt(s_1))}{\paramt(s_2) - \paramt(s_1)}}
        \leq \frac{1}{c}, \quad 0 \leq \forall s_1 \leq \forall s_2 \leq S.
    \end{align*}
    This ends the check of (ii).

    Finally, letting $\tilde{u} \coloneqq u \circ \paramt$, 
    we get
    \begin{align*}
        d(\tu(s_1), \tu(s_2)) \leq \int_{\paramt(s_1)}^{\paramt(s_2)} |u'|(r)dr
        = s_2 - s_1, \quad 0 \leq \forall s_1 \leq \forall s_2 < \sa_u,
    \end{align*}
    which proves (iii).
\end{proof}

\begin{prop}
\label{prop:property_of_arc_length_param}
The arc-length parametrization $\tilde{u}$ satisfies the following:
    \begin{enumerate}
        \item $f \circ \tilde{u}$ is locally absolutely continuous on $[0, \sa_u)$;
        \item $(f\circ \ut)'(s) = - \localslope \circ \ut (s)$ for  $\ms{L}\text{-a.e.}$ in $s \in [0, \sa_u)$.
    \end{enumerate}
\end{prop}

\begin{proof}
    \textit{(i)}:
    Since the local slope is a strong upper gradient,
    the change of variables for $\paramt$ gives for any $s_0, s_1 \in [0, \sa_u)$ with $s_0 \leq s_1$
    \begin{align}
        \label{eq:arc_length_param_1}
        |f \circ \ut (s_0) - f\circ \ut(s_1)|
        \leq \int_{\paramt(s_0)}^{\paramt(s_1)} \localslope \circ u(r) \cdot |u'|(r) dr
        \underset{\scriptsize \eqref{eq: differential_of_paramt}}{=} \int_{s_0}^{s_1} \localslope \circ \ut (r)dr.
    \end{align}
    The second equation in~\eqref{eq:arc_length_param_1} and Proposition~\ref{prop:p_cms_energy_identity} imply that $\localslope \circ \ut \in L^1_{loc}([0, \sa_u))$; therefore (i) holds.

    \textit{(ii)}:
    Let $E$ be all points in $[0, \ta{u})$ at which the energy identity~\eqref{eq:energy_identity} does not hold, and let $N$ be all points in $[0, \sa_u)$ at which \eqref{eq: differential_of_paramt} does not hold, thus $N$ is $\ms{L}$-negligible.
    The set $\paramt^{-1}(E) = \params(E)$ is also $\ms{L}$-negligible because $\params$ is locally absolutely continuous on $[0, \ta{u})$ and $E$ is $\ms{L}$-negligible.
    The chain rule gives $(f \circ \tu)'(s) = - \localslope \circ \tu(s)$ for any $s \in [0, \sa_u) \setminus \left( \paramt^{-1}(E) \cup N \right)$.
    Thus we conclude.
\end{proof}

Utilizing the arc-length parametrization, we can show the convexity of $f$ along with $p$-curves of maximal slope.
\begin{lem}[Convexity of $f$ along with $p$-curves of maximal slope]
    \label{lem:convexity_of_f_circ_u_tilde} 
    Let one of the following hold:
    \begin{enumerate}[label=(\alph*), nolistsep]
        \item $\lambda=0$, $p_0 \in (1, +\infty)$ and $p \in (1, +\infty)$;
        \item $\lambda < 0$, $p_0 \in [2, +\infty)$ and $p \in (1, p_0]$.
    \end{enumerate}
    Suppose that $f$ is $(p_0, \lambda)$-convex, and let $u \in \pcms{p}{a}$.
    Then for any interval $I \subset [0, \sa_u)$ with the width $|I| \leq 1$ the map $s \mapsto f \circ \tu(s)$ is $(2, -\lambda^-)$-convex on $I$, i.e., for all $s_0, s_1 \in I$, letting $s_\theta \coloneqq (1 - \theta) s_0 + \theta s_1$, we have
    \begin{align*}
        f \circ \tu(s_\theta) \leq (1-\theta)f\circ \tu(s_0) + \theta f\circ \tu(s_1) + \lambda^{-}\theta(1-\theta)|s_1 - s_0|^2
    \end{align*}
    for every $\theta \in [0, 1]$, where $\tu$ is the arc-length parametrization of $u$ and the constant $\sa_u$ is as in Lemma~\ref{lem:arc_length_param}.
\end{lem}

\begin{proof}
    Suppose that $f \circ \tu$ is not $(2, -\lambda^-)$-convex on $I$.
    There exist $s_0, s_1 \in I$ with $s_0 < s_1$ and $\theta_0 \in [0, 1]$ such that
    \begin{align}
        \label{eq:convexity_of_f_along_with_pcms_contradiction}
        f\circ \tu(s_{\theta_0}) > \varphi(\theta_0),
    \end{align}
    where $s_\theta \coloneqq (1-\theta)s_0 + \theta s_1$ and $\varphi(\theta) \coloneqq (1 - \theta) f\circ \tu(s_0) + \theta f \circ \tu (s_1) + \lambda^- \theta(1-\theta)|s_0-s_1|^2$ for any $\theta \in [0, 1]$.
    Letting
    \begin{align*}
        g(\theta) \coloneqq f\circ \tu(s_\theta) - \varphi(\theta),
    \end{align*}
    the function $[0,1] \ni \theta \mapsto g(\theta)$ is absolutely continuous because $f \circ \tu$ is so.
    We can take three points $\theta_l, \theta_\ast, \theta_r \in [0, 1]$ with $\theta_l < \theta_\ast \leq \theta_r$ such that $g(\theta_\ast) = \max_{\theta} g(\theta) > 0$,
    $g'(\theta_l) > 0$,
    $g(\theta_l) \in [0, g(\theta_\ast)]$, 
    and $g(\theta_r) = g(\theta_l)$,
    because of~\eqref{eq:convexity_of_f_along_with_pcms_contradiction}, the continuity of $g$ and the intermediate value theorem.
    It follows from $g'(\theta_l) > 0$ and a direct calculation that
    \begin{align}
        \label{eq:convexity_of_f_along_with_pcms_1}
        (s_1 - s_0)\localslope \circ \tu (s_{\theta_l}) < f\circ \tu(s_0) - f\circ \tu(s_1) - \lambda^-(1 - 2\theta_l)|s_0-s_1|^2.
    \end{align}
    By adding $-(\theta_r - \theta_l)\lambda^-(1-2\theta_l)|s_0-s_1|^2$ to both sides of $g(\theta_l)=g(\theta_r)$, one sees that
    \begin{align}
        \begin{split}
        \label{eq:convexity_of_f_along_with_pcms_2}
        (\theta_r - \theta_l)\left( f\circ \tu(s_0) - f\circ \tu(s_1) - \lambda^-(1-2\theta_l)|s_0 - s_1|^2 \right)\\
        = f\circ \tu(s_{\theta_l}) - f\circ \tu(s_{\theta_r}) -\lambda^-(\theta_r-\theta_l)^2|s_0 - s_1|^2.
        \end{split}
    \end{align}
    Recalling that $|I| \leq 1$, $1$-Lipschitz continuity of $\tu$ yields
    \begin{align}
        \label{eq:convexity_of_f_along_with_pcms_3}
        d(\tu(s_{\theta_l}), \tu(s_{\theta_r})) \leq (\theta_r - \theta_l)(s_1 - s_0) \leq 1.
    \end{align}
    Note that $\lambda^-(\theta_r - \theta_l)^2|s_0 - s_1|^2 = 0$ in~\eqref{eq:convexity_of_f_along_with_pcms_2} for the case (a).
    Thus we can assume that $p_0 \geq 2$, and then we have
    \begin{align*}
        d^{p_0}(\tu(s_{\theta_l}), \tu(s_{\theta_r}))
        \leq (\theta_r - \theta_l)^{p_0}(s_1 - s_0)^{p_0}
        \leq (\theta_r - \theta_l)^2(s_1 - s_0)^2.
    \end{align*}
    Combining \eqref{eq:convexity_of_f_along_with_pcms_1} and \eqref{eq:convexity_of_f_along_with_pcms_2} with this inequality, the global formula of the local slope gives
    \begin{align*}
        \localslope \circ \tu (s_{\theta_l})
        < \localslope \circ u(s_{\theta_l}) \frac{d(\tu(s_{\theta_l}), \tu(s_{\theta_r}))}{(s_1 - s_0)(\theta_r - \theta_l)}
        \leq \localslope \circ \tu (s_{\theta_l}).
    \end{align*}
    This is a contradiction.
\end{proof}

Recalling that a convex function $g \colon [a, b] \to \mb{R}$ is differentiable at $\ms{L}$-a.e. points in $(a, b)$ and its derivative bounded from below locally in $(a, b)$.
If $\phi$ is $(2, -\lambda^-)$-convex on $[a, b]$, then the map $x \mapsto \phi(x) + \lambda^-x^2$ is convex on $[a, b]$.
Therefore $\phi$ is differentiable at $\ms{L}$-a.e. points in $(a, b)$, and its derivative $\phi'$ is bounded from below locally. 
This fact and the local $(2, -\lambda^-)$-convexity of $f\circ \tu$ yield the local slope's bound from above along $p$-curves of maximal slope.
Namely:
\begin{lem}[Slope's upper bound along with $p$-curves of maximal slope]
\label{lem:upper_bound_of_local_slope}
    Let one of the following hold:
    \begin{enumerate}[label=(\alph*), nolistsep]
        \item $\lambda=0$, $p_0 \in (1, +\infty)$ and $p \in (1, +\infty)$;
        \item $\lambda < 0$, $p_0 \in [2, +\infty)$ and $p \in (1, p_0]$.
    \end{enumerate}
    Suppose that $f$ is $(p_0, \lambda)$-convex, and let $u \in \pcms{p}{a}$.
    Then for any $[T_1, T_2] \subset (0, \ta{u})$ there exists a constant $M > 0$ such that
    \begin{align}
        \label{eq:upper_bound_of_local_slope}
        \localslopeu (t) \leq M \quad \text{for every } t \in [T_1, T_2].
    \end{align}
    Moreover, if $\lambda = 0$, then the map $t \mapsto |u'|(t)$ is $\ms{L}\text{-a.e.}$ non-increasing in $[0, \ta{u})$.
\end{lem}

\begin{proof}
    Let $J \coloneqq [T_1, T_2] \subset (0, \ta{u})$, and let $\paramt$ be the arc-length parameter and $\params$ be the inverse map of $\paramt$.
    By the continuity of $\params$, we can cover $\params(J)$ by finite intervals $\{I_n\}_{n=1}^N$ such that $I_n \subset (0, \sa_u)$ and the width $|I_n| \leq 1$.
    The $(2, -\lambda^-)$-convexity of $f \circ \tu$ on each $I_n$ gives for each $I_n$ a constant $M_n \in \mb{R}$ satisfying
    \begin{align*}
        - \localslope \circ \tu (s) 
        \underset{\scriptsize \text{Prop.}~\ref{prop:property_of_arc_length_param}}{=}(f\circ\tu)'(s) \geq M_n \quad \text{for } \ms{L}\text{-a.e. in } s \in I_n;
    \end{align*}
    hence letting $M \coloneqq \max\{-M_1, \ldots, -M_N\}$ we have
    \begin{align*}
        \localslope \circ \tu (s) \leq M \quad \text{for } \ms{L}\text{-a.e. in } s \in \params(J).
    \end{align*}
    Let $E$ be the set of all points in $\params(J)$ at which this inequality does not hold, thus
    \begin{align*}
        \localslopeu(t) \leq M \quad \text{for every } t \in J \setminus \params^{-1}(E).
    \end{align*}
    The set $\params^{-1}(E) = \paramt(E)$ is $\ms{L}$-negligible because $E$ is so and $\paramt$ is locally absolutely continuous.
    This completes the proof of \eqref{eq:upper_bound_of_local_slope}.

    If $\lambda = 0$, we can see that $f \circ \tu$ is convex on $[0, \sa_u)$ by gluing the local convexity of it.
    Therefore $\localslope \circ \tu$ is $\ms{L}\text{-a.e.}$ non-increasing on $[0, \sa_u)$ from Proposition~\ref{prop:property_of_arc_length_param}.
    Letting $E \subset [0, \sa_u)$ be that $\localslope \circ \tu$ is non-increasing on $E$, we can see that $\localslope \circ u $ is non-increasing on $[0, \ta{u}) \setminus \params^{-1}(E^c)$.
    The set $\params^{-1}(E^c)$ is $\ms{L}$-negligible because $E^c$ is so and $\params^{-1}=\paramt$ is locally absolutely continuous;
    thus the energy identity~\eqref{eq:energy_identity} gives the desired result.
\end{proof}

\section{Proofs of main theorems}
\label{sec:proofs_of_main_theorems}
This section provides the proof of the main theorems: Theorems~\ref{thm:p_cms_to_p_dash_cms_lambda_gt_0}, \ref{thm:p_cms_to_p_dash_cms_lambda_st_0} and \ref{thm:regularizing_effects}.
Let $(\ms{S},d)$ be a complete metric space, $f \colon \ms{S} \to (-\infty, +\infty]$ be a proper and lower semicontinuous, and $u_0 \in D(f)$ throughout this section.
In what follows, without loss of generality, we can assume that $\lambda \leq 0$ excluding Theorem~\ref{thm:regularizing_effects}.

We begin by constructing a parameter transformation for $p$-curves of maximal slope whose domain is not necessarily unbounded.
\begin{lem}[Parameter transformation from $p$-curve of maximal slope to $p'$-curve of maximal slope]
    \label{lem:param_trans_domain_possibly_bounded}
    Let one of the following hold:
    \begin{enumerate}[label=(\alph*)]
        \item $\lambda=0$, $p_0 \in (1, +\infty)$ and $p \in (1, +\infty)$;
        \item $\lambda < 0$, $p_0 \in [2, +\infty)$ and $p \in (1, p_0]$.
    \end{enumerate}
    Suppose that $f$ is $(p_0, \lambda)$-convex, and let $u \in \pcms{p}{a}$.
    Then for any exponent $p' \in (1, +\infty)$ if $\lambda = 0$ or any exponent $p' \in (1, p_0]$ if $\lambda < 0$ there exist a strictly increasing absolutely continuous map $\tilde{\params}_{p \to p'} \colon [0, \ta{u}) \to [0, \sa_{p \to p'})$ and its inverse map $\tilde{\paramt}_{p \to p'} \colon [0, \sa_{p \to p'}) \to [0, \ta{u})$ with the following: 
    \begin{enumerate}
        \item the map $t \mapsto \tilde{\params}_{p \to p'}(t)$ is locally bi-Lipschitz on $(0, \ta{u})$ and satisfies
        \begin{align}
            \label{eq:derivative_of_params}
            \tilde{\params}'_{p \to p'}(t) = |u'|^\alpha(t) > 0,\quad \ms{L}\text{-a.e. in } t \in [0, \ta{u});
        \end{align}
        \item the map $s \mapsto \tilde{\paramt}_{p \to p'}(s)$ is locally bi-Lipschitz on $(0, \sap{p}{p'})$ and satisfies 
        \begin{align}
            \label{eq:derivative_of_paramt}
            \tilde{\paramt}'_{p \to p'}(s) = |u'|^{-\alpha}(\tilde{\paramt}_{p \to p'}(s)),\quad \ms{L}\text{-a.e. in } s \in [0, \sap{p}{p'});
        \end{align}
        \item the curve $u_{p'} \coloneqq u \circ \tilde{\paramt}_{p \to p'}$ belongs to $\text{CMS}_{p'}([0, \sap{p}{p'}); u_0)$;
    \end{enumerate}
    where the exponent $\alpha$ is defined to be
    $
        \alpha \coloneqq 1 - \frac{p}{q} \cdot \frac{q'}{p'}  
    $,
    and the constant $\sap{p}{p'}$ is defined by 
    $ \sap{p}{p'} \coloneqq \int_{0}^{\ta{u}} |u'|^\alpha(r)dr $.
    \vspace{1.5mm}

    Moreover, by considering the transformations $\tilde{\params}_{p'\to p}$ and $\tilde{\paramt}_{p' \to p}$ of $u_{p'}$, we have
    \begin{enumerate}
        \setcounter{enumi}{3}
        \item (dual relation) $\ta{u_{p'}} = \sap{p}{p'}$ and $\sap{p'}{p} = \ta{u}$;
        \item (inverse relation) $\tilde{\params}_{p' \to p} = \tilde{\paramt}_{p \to p'}$ and $\tilde{\paramt}_{p' \to p} = \tilde{\params}_{p \to p'}$.
    \end{enumerate}
\end{lem}

\begin{proof}
    We define the map $\tilde{\params}_{p \to p'} \colon [0, \ta{u}) \to [0, \sap{p}{p'})$ to be
    \begin{align*}
        \label{eq:bounds_derivative_of_params}
        \tilde{\params}_{p \to p'}(t) \coloneqq \int_{0}^t |u'|^\alpha(r)dr.
    \end{align*}

    Let us check (i) and (ii).
    Note that $\alpha \in (-\infty, 1)$.
    In the case $\alpha \geq 0$ (resp. $\alpha < 0$), the local integrability of $|u'|$ - Proposition~\ref{prop:p_cms_energy_identity}- (resp. the local positive lower bound of $\localslope \circ u$  -Corollary~\ref{cor:lower_bound_of_local_slope}-) implies that $\tilde{\params}$ is locally absolutely continuous on $[0, \ta{u})$.
    For any $[T_1, T_2] \subset (0, \ta{u})$ the local lower and upper bound of $\localslope \circ u$ and the energy identity~\eqref{eq:energy_identity} show the existence of constants $c_0, c_1 > 0$ such that
    \begin{align}
        c_0 \leq |u'|^\alpha(t) \leq c_1,\quad \ms{L}\text{-a.e. in } t \in [T_1, T_2];
    \end{align}
    therefore the map $t \mapsto \tilde{\params}_{p \to p'}(t)$ is strictly increasing and locally bi-Lipshitz on $(0, \ta{u})$.
    The inequality~\eqref{eq:bounds_derivative_of_params} implies \eqref{eq:derivative_of_params}.
    The inverse map $\tilde{\paramt}_{p \to p'} \colon [0, \sap{p}{p'}) \to [0, \ta{u})$ of $\tilde{\params}_{p \to p'}$ is also locally bi-Lipschitz on $(0, \sap{p}{p'})$.
    Let $E \subset [0, \ta{u})$ be the set of all points at which \eqref{eq:derivative_of_params} does not hold, and let $N \subset [0, \sap{p}{p'})$ be the set of all points at which $\tilde{\paramt}_{p \to p'}$ is not differentiable.
    The set $N$ is clearly $\ms{L}$-negligible, and $\tilde{\paramt}^{-1}_{p \to p'}(E) = \tilde{\params}_{p \to p'}(E)$ is so because $\tilde{\params}_{p\to p'}$ is locally absolutely continuous and $E$ is  $\ms{L}$-negligible.
    For any $s \in [0, \sap{p}{p'}) \setminus (\tilde{\paramt}^{-1}_{p \to p'}(E) \cup N)$ the inverse function rule gives
    \begin{align*}
            \tilde{\paramt}'_{p \to p'}(s) = |u'|^{-\alpha}(\tilde{\paramt}_{p \to p'}(s)).
    \end{align*}
    This completes our discussion of (i) and (ii).

    To show (iii), let us check that the curve $u_{p'} \coloneqq u \circ \tilde{\paramt}_{p \to p'} \colon [0, \sap{p}{p'}) \to \ms{S}$ satisfies the conditions as $p'$-CMS.
    It is clear that $f \circ u_{p'} \colon [0, \sap{p}{p'}) \to \mb{R}$ is non-increasing.
    To check \eqref{eq:p_cms}, it is enough to show that the energy identity
    \begin{align}
        \label{eq:paramtrans_energy_identity}
        (f \circ u_{p'})'(s) = -\localslope \circ u_{p'}(s) \cdot |u_{p'}'|(s) = - |u_{p'}'|^{p'}(s) = -\localslope^{q'} \circ u_{p'} (s)
    \end{align}
    holds for $\ms{L}\text{-a.e. } s$ in $[0,\sap{p}{p'})$.
    Let $E_1$ be the set of all points in $[0, \ta{u})$ at which the energy identity~\eqref{eq:energy_identity} does not hold, and let $N_1$ be the set of all points in $[0, \sap{p}{p'})$ at which $\tilde{\paramt}_{p \to p'}$ does not satisfy \eqref{eq:derivative_of_paramt}.
    The set $N \coloneqq \tilde{\paramt}_{p \to p'}^{-1}(E_1) \cup N_1$ is $\ms{L}$-negligible because $\tilde{\params}_{p \to p'}$ sends $\ms{L}$-negligible sets to ones.
    For any point $s \in [0, \sap{p}{p'}) \setminus N$, \eqref{eq:energy_identity} gives \eqref{eq:paramtrans_energy_identity} because of
    \begin{align}
        |u_{p'}'|(s)
        &= \lim_{s'\to s}\frac{d(u(\tilde{\paramt}_{p \to p'}(s)), u(\tilde{\paramt}_{p \to p'}(s'))}{|\tilde{\paramt}_{p \to p'}(s)-\tilde{\paramt}_{p \to p'}(s')|}\cdot\frac{|\tilde{\paramt}_{p \to p'}(s) - \tilde{\paramt}_{p \to p'}(s')|}{|s - s'|} \notag \\
        &= |u'|(\tilde{\paramt}_{p \to p'}(s)) \cdot |u'|^{-\alpha}(\tilde{\paramt}_{p \to p'}(s)) \label{eq:u_tilde_and_u_prime}\\
        &= |u'|^{\frac{p}{q}\cdot\frac{q'}{p'}}(\tilde{\paramt}_{p \to p'}(s))
        = \localslope^{\frac{q'}{p'}}\circ u_{p'} (s) \notag
    \end{align}
    and
    \begin{align*}
        (f\circ u_{p'})'(s) = (f \circ u)'(\tilde{\paramt}_{p \to p'}(s))\cdot \tilde{\paramt}_{p \to p'}'(s)
        = -\localslope\circ u_{p'}(s) \cdot |u_{p'}'|(s).
    \end{align*}
    The monotonicity and the local Lipschitz continuity of $\tilde{\paramt}_{p \to p'}$ allow us to apply the change of variables; hence for any $s_1, s_2 \in [0, \sap{p}{p'})$ with $s_1 \leq s_2$
    \begin{align}
        \label{eq:check_ut_acloc}
        d(u_{p'}(s_1), u_{p'}(s_2))
        \leq \int_{\tilde{\paramt}_{p \to p'}(s_1)}^{\tilde{\paramt}_{p \to p'}(s_2)}|u'|(r)dr
        \underset{\eqref{eq:u_tilde_and_u_prime}}{=} \int_{s_1}^{s_2} |u_{p'}'|(r)dr.
    \end{align}
    Since the map $s \mapsto |u_{p'}'|(s)$ belongs to $L_{loc}^1([0, \sap{p}{p'}))$, $u_{p'}$ is a member of $AC_{loc}([0, \sap{p}{p'}); \ms{S})$.
    This completes the proof of (iii) in Lemma~\ref{lem:param_trans_domain_possibly_bounded}.

    Finally, let us check that (iv) and (v) hold.
    It is clear that $\ta{u_{p'}} = \sap{p}{p'}$.
    Letting $\beta \coloneqq 1 - \frac{p'}{q'}\cdot\frac{q}{p}$ one sees that for every $t \in [0, \ta{u_{p'}})$
    \begin{align*}
        &\tilde{\params}_{p' \to p}(t)
        \coloneqq \int_{0}^t |u_{p'}'|^{\beta}(r)dr
        \underset{\eqref{eq:u_tilde_and_u_prime}}{=} \int_0^t |u'|^{(1-\alpha)(\beta)}(\tilde{\paramt}_{p \to p'}(r))dr\\
        &= \int_0^t |u'|^{-\alpha}(\tilde{\paramt}_{p \to p'}(r))dr
        \underset{\text{(ii)}}{=} \int_0^t \tilde{\paramt}'_{p \to p'}(r)dr
        \underset{\text{(ii)}}{=} \tilde{\paramt}_{p \to p'}(t).
    \end{align*}
    Thus we have $\sap{p'}{p} = \ta{u}$ and the inverse map $\tilde{\paramt}_{p' \to p}$ of $\tilde{\params}_{p' \to p}$ is equal to $\tilde{\params}_{p \to p'}$.
\end{proof}

We are now in a position to prove Theorem~\ref{thm:p_cms_to_p_dash_cms_lambda_gt_0}.
Note that the constant $\ta{u}$ in this theorem is given by \eqref{eq:t_ast_u}.

\begin{proof}[proof of Theorem~\ref{thm:p_cms_to_p_dash_cms_lambda_gt_0}]
We write $\sap{p}{p'}$ in Lemma~\ref{lem:param_trans_domain_possibly_bounded} as $S^\ast_u$ for simplicity.
Let $\paramt_{p \to p'} \coloneqq \tilde{\paramt}_{p\to p'} \colon [0, S^\ast_u) \to [0, \ta{u})$ as in Lemma~\ref{lem:param_trans_domain_possibly_bounded}.

Firts, we extend $\paramt_{p \to p'}$ over $[0, +\infty)$, and check (i) and (ii) in Theorem~\ref{thm:p_cms_to_p_dash_cms_lambda_gt_0}.
We divide the proof into the following four cases, (A) - (D).

\begin{proof}[(A) The case $S^\ast_u = +\infty$ and $\ta{u} = +\infty$]
    In this case, the map $s \mapsto \paramt_{p\to p'}(s)$ is already defined over $[0, +\infty)$.
    From Lemma~\ref{lem:param_trans_domain_possibly_bounded}, we have (i) and (ii).
    \renewcommand{\qedsymbol}{}
\end{proof}

\begin{proof}[(B) The case $S^\ast_u = +\infty$ and $\ta{u} < +\infty$]
    Similarly as case (A), (i) and (ii) hold; however we extend $\paramt_{p\to p'}$ over $[0, +\infty]$ by $\paramt_{p \to p'}(+\infty) \coloneqq \ta{u}$, which is needed in the proof of (iii).
    \renewcommand{\qedsymbol}{}
\end{proof}

\begin{proof}[(C) The case $S^\ast_u < +\infty$ and $\ta{u} < +\infty$]
    We extend $\paramt_{p\to p'}$ over $[0, +\infty)$ by satisfying 
    \begin{align*}
     \paramt_{p \to p'}(s) \coloneqq \ta{u} \quad \text{for any } s \in [S^\ast_u, +\infty).    
    \end{align*}
    It follows from Lemma~\ref{lem:param_trans_domain_possibly_bounded} that (ii) holds.
    Let us check (i).
    It is clear that $f\circ u_{p'}$ is non-increasing.
    To check \eqref{eq:p_cms}, it is enough to show that the energy identity
    \begin{align*}
        (f\circ u_{p'})'(s) = - \localslope \circ u_{p'}(s) |u_{p'}'|(s) = - |u_{p'}'|^{p'}(s) = -\localslope^{q'} \circ u_{p'}(s)
    \end{align*}
    holds for $\ms{L}\text{-a.e.}$ in $s \in [S^\ast_u, +\infty)$;
    however this equality holds for any $s \in [S^\ast_u, +\infty)$ since $\localslope \circ u_{p'}(s) = \localslope \circ u(\ta{u}) = 0$ and $u_{p'}$ is stationary on $[S^\ast_u, +\infty)$.
    Finally, we show that $u_{p'} \in AC_{loc}([0, +\infty); \ms{S})$.
    Since $\tilde{\paramt}_{p \to p'}$ is locally Lipshictz on $(0, \sa_u)$, the change of variable yields that for any $s_1, s_2 \in (0, \sa_u)$ with $s_1 \leq s_2$
    \begin{align*}
        d(u_{p'}(s_1), u_{p'}(s_2))
        \leq \int_{\paramt_{p \to p'}(s_1)}^{\paramt_{p\to p'}(s_2)} |u'|(r)dr
        = \int_{s_1}^{s_2}|u_{p'}'|(r)dr.
    \end{align*}
    Since $u_{p'}$ is stationary on $[S^\ast_u, +\infty)$, the monotone convergence theorem gives for any $s_1, s_2 \in [0, +\infty)$ with $s_1 \leq s_2$
    \begin{align*}
        d(u_{p'}(s_1), u_{p'}(s_2)) \leq \int_{s_1}^{s_2} |u_{p'}'|(r)dr;
    \end{align*}
    in addition the local integrability of $|u'|$ on $[0, +\infty)$ shows $|u_{p'}'| \in L^1([0, +\infty))$.
    This completes the proof of (i).
    \renewcommand{\qedsymbol}{}
\end{proof}

\begin{proof}[(D) The case $S^\ast_u < +\infty$ and $\ta{u} = +\infty$]
    We extend $\paramt_{p\to p'}$ over $[0, +\infty)$ by satisfying
    \begin{align*}
        \paramt_{p \to p'}(s) \coloneqq \ta{u} \quad \text{for any } s \in [S^\ast_u, +\infty).    
    \end{align*}
    Then it is trivial that the statement (ii) holds.

    To define $u_{p'}$ on $[S^\ast_u, +\infty)$ and check (i), we show that 
    \begin{align}
        \begin{split}
            \label{cond:existence_point_at_infinity}
            &|u'| \in L^1([0, +\infty)),
            \text{ there exists } u(\ta{u}) \coloneqq \lim_{t \:\uparrow\: \ta{u}}u(t)\\
            &\text{ and } \localslope (u(\ta{u})) = 0.
        \end{split}
    \end{align}
    The map $t \mapsto |u'|(t)$ is $\ms{L}\text{-a.e.}$ non-increasing -Lemma~\ref{lem:upper_bound_of_local_slope}- and $S^\ast_u < + \infty$ so that $\alpha > 0$, which gives
    \begin{align}
        \label{eq:metric_derivative_to_zero}
        \liminf_{t\to+\infty}|u'|(t) = 0.
    \end{align}
    Hence there exists $t_0 \in [0, +\infty)$ such that $|u'|(t) \leq 1$ for $\ms{L}\text{-a.e.}$ in $t \in [t_0, +\infty)$.
    Note that $\alpha < 1$ and $\sa_u < +\infty$ the local integrability of $|u'|$ yields
    \begin{align*}
        \int_{0}^{+\infty} |u'|(r)dr \leq \int_{0}^{t_0} |u'|(r) dr + \int_{t_0}^{+\infty} |u'|^\alpha(r)dr < +\infty.
    \end{align*}
    This implies that $u$ is absolutely continuous on $[0, +\infty)$; therefore there exists $u(\ta{u})$ because $\ms{S}$ is complete.
    Using the lower semicontinuity of $\localslope \circ u$ and \eqref{eq:metric_derivative_to_zero}, the energy identity~\eqref{eq:energy_identity} induces $\localslope(u(\ta{u}))=0$. 
    This ends the check of~\eqref{cond:existence_point_at_infinity}.

    From \eqref{cond:existence_point_at_infinity}, we can extend $u$ over $[0, \ta{u}]$; thus we can define $u_{p'} \coloneqq u \circ \paramt_{p \to p'}$ on $[0, +\infty)$.
    It is clear that $f \circ u_{p'}$ is non-increasing.
    Similarly as in case (C), the facts that $u_{p'}$ is stationary and $\localslope \circ u_{p'} \equiv  0$ on $[S^\ast_u, +\infty)$ induce \eqref{eq:p_cms}, and the fact $|u'| \in L^1([0, +\infty))$ yields $u_{p'} \in AC([0, +\infty); \ms{S})$.
    This completes the proof of (i).
    \renewcommand{\qedsymbol}{}
\end{proof}

In the above, we end the check of the statements (i) and (ii) in Theorem~\ref{thm:p_cms_to_p_dash_cms_lambda_gt_0}.
Finally, let us prove (iii): we show it only in case (B) because we can prove it similarly in the other cases.

From the definition of $u_{p'}$, we can easily check that
\begin{align*}
    \ta{u_{p'}} = S^\ast_u = +\infty.
\end{align*}
From the definition of $u_{p'}$, we have
\begin{align*}
    S^\ast_{u_{p'}} \coloneqq S^\ast_{p' \to p, u_{p'}} = \ta{u} < +\infty
\end{align*}
and there exists the limit
\begin{align*}
    \lim_{t \:\uparrow\: \ta{u_{p'}}}u_{p'}(t) = \lim_{t \:\uparrow\: S^\ast_u} u \circ \paramt_{p \to p'}(t) = u(\ta{u}).
\end{align*}
Hence we are now in case (D) and the parameter transformation $\paramt_{p' \to p}$ for $u_{p'}$ is defined to be
\begin{align*}
    \paramt_{p' \to p}(s) = \begin{cases}
        \tilde{\paramt}_{p' \to p}(s) & \text{ on } [0, S^\ast_{u_{p'}})\\
        \ta{u_{p'}} & \text{ on } [S^\ast_{u_{p'}}, +\infty)
    \end{cases},
\end{align*}
where $\tilde{\paramt}_{p' \to p}$ is as in Lemma~\ref{lem:param_trans_domain_possibly_bounded}.
A simple calculation yields
\begin{align*}
    \paramt_{p' \to p}([0, S^\ast_{u_{p'}})) = \tilde{\paramt}_{p' \to p}([0, S^\ast_{u_{p'}})) = [0, \ta{u_{p'}}) = [0, S^\ast_u);
\end{align*}
therefore we have for any $s \in [0, S^\ast_{u_{p'}})$
\begin{align*}
    u_{p'}\circ \paramt_{p' \to p}(s) = u \circ \tilde{\paramt}_{p \to p'} \circ \tilde{\paramt}_{p' \to p}(s) = u(s).
\end{align*}
Moreover we get for any $s \in [S^\ast_{u_{p'}}, +\infty) = [\ta{u}, +\infty)$
\begin{align*}
    u_{p'}\circ \paramt_{p' \to p}(s) = u \circ \paramt_{p \to p'}(S^\ast_u) = u(\ta{u}) \underset{\scriptsize \text{Cor.}~\ref{cor:lower_bound_of_local_slope}}{=} u(s).
\end{align*}
This completes our check of (iii).
\end{proof}

Next, let us give the proof of Theorem~\ref{thm:p_cms_to_p_dash_cms_lambda_st_0}, which has many common parts with the proof of Theorem~\ref{thm:p_cms_to_p_dash_cms_lambda_gt_0}.
We also note that, in Theorem~\ref{thm:p_cms_to_p_dash_cms_lambda_st_0}, the constant $\ta{u}$ is \eqref{eq:t_ast_u} and $\sa_{p \to p'}$ is as in Lemma~\ref{lem:param_trans_domain_possibly_bounded}.
\begin{proof}[proof of Theorem~\ref{thm:p_cms_to_p_dash_cms_lambda_st_0}]
    We denote $\sap{p}{p'}$ by $S^\ast_u$ for simplicity.

    First, let us check (i) and (ii).
    Let $\paramt_{p \to p'} \coloneqq \tilde{\paramt}_{p\to p'} \colon [0, S^\ast_u) \to [0, \ta{u})$ as in Lemma~\ref{lem:param_trans_domain_possibly_bounded}.
    By the similar argument as the proof of Theorem~\ref{thm:p_cms_to_p_dash_cms_lambda_gt_0}, in the following cases
    \begin{enumerate}[label=(\Alph*)]
        \item $\sa_u = +\infty$ and $\ta{u} = +\infty$;
        \item $\sa_u = +\infty$ and $\ta{u} < +\infty$;
        \item $\sa_u < +\infty$ and $\ta{u} < +\infty$;
    \end{enumerate}
    we can extend $\paramt_{p \to p'}$ over $[0, +\infty)$ and check (i) and (ii).
    Furthermore, by the same way as the proof of Theorem~\ref{thm:p_cms_to_p_dash_cms_lambda_gt_0},
    in case (D) $\sa_u < +\infty$ and $\ta{u} = +\infty$,
    we can extend $\paramt_{p \to p'}$ over $[0, +\infty)$ and show that (ii) holds.
    Thus let us check (i) in case (D).

    To define $u_{p'}$ on $[S^\ast_u, +\infty)$ and check (i), we show that 
    \begin{align}
        \begin{split}
            \label{cond:existence_point_at_infinity_lambda_st_0}
            |u'| \in L^1([0, +\infty)) \text{ and } \localslope (u(\ta{u})) = 0.
        \end{split}
    \end{align}
    Note that the existence of $u(\ta{u})$ is assumed in Theorem~\ref{thm:p_cms_to_p_dash_cms_lambda_st_0},
    combining this and the lower semicontinuity of $f$, we have
    \begin{align*}
        \int_0^{+\infty} |u'|^p(r) dr \underset{\eqref{eq:energy_identity}}{=} \limsup_{t \to +\infty} \left(f\circ u(0) - f\circ u (t) \right)
        \leq f\circ u(0) - f(u(\ta{u})) < +\infty.
    \end{align*}
    Since $\alpha < 1$ and $p > 1$, one sees that
    \begin{align*}
        \int_{0}^{\infty} |u'|(r)dr
        \leq \int_{\{|u'| \leq 1\}} |u'|^\alpha(r) dr + \int_{\{|u'| > 1\}} |u'|^p(r)dr < +\infty;
    \end{align*}
    thus we get $\liminf_{t\to+\infty}|u'|(t) = 0$, and then $\localslope(u(\ta{u})) = 0$.
    This completes the verification of~\eqref{cond:existence_point_at_infinity_lambda_st_0}.
    It follows from \eqref{cond:existence_point_at_infinity_lambda_st_0} that (i) holds by the same discussion as in the proof of Theorem~\ref{thm:p_cms_to_p_dash_cms_lambda_gt_0}.

     The statement (iii) can be checked in the same way as in the proof of Theorem~\ref{thm:p_cms_to_p_dash_cms_lambda_gt_0}.
\end{proof}

The existence assumption of $\lim_{t \:\uparrow\: +\infty} u(t)$ in Theorem~\ref{thm:p_cms_to_p_dash_cms_lambda_st_0} (c) cannot be dropped because of the following example:
\begin{ex}
    \label{ex:assumption_cannot_be_dropped}
    Let $(\ms{S}, d)$ be a Euclidian space $(\mb{R}, d)$ and $f(x) \coloneqq -\frac{1}{2}x^2$.
    This function $f$ is $(2, -2)$-convex, and
    the curve $u \colon [0, +\infty) \to \mb{R}$ defined by $u(t) \coloneqq e^t$ belongs to $\pcms{2}{+\infty}$ with $u_0 = 1$.

    Let $p' \in (1, 2)$.
    Applying Lemma~\ref{lem:param_trans_domain_possibly_bounded}, we can easily check that $u_{p'}(s) = (1 + \alpha s)^{1/\alpha}$ and its domain is $[0, -\frac{1}{\alpha})$, where $\alpha = 1 - \frac{q'}{p'} < 0$ and $q'$ is the conjugate exponent of $p'$.
    Therefore $u_{p'}$ blows up in finite time and cannot be extended to infinite time.
\end{ex}
\noindent
It is also worth mentioning that this example cannot be obtained by minimizing movements scheme, discussed in~\cite[Chapter~2 and 3]{ambrosioGradientFlowsMetric2008}, because $f$ is not $p$-coercive defined in the sense of~\cite[(2.1.2b)]{ambrosioGradientFlowsMetric2008} for $p \in (1, 2)$.

Finally, let us prove the remaining main result -Theorem~\ref{thm:regularizing_effects}-.
Note that this theorem relaxes some assumptions imposed in previous results, for example, the assumption that a $p$-curve of maximal slope for $(p, \lambda)$-convex functions with the same $p$ and some $\lambda \in \mb{R}$ is given by discrete approximations called generalized minimizing movements in~\cite[Theorem~2.4.15]{ambrosioGradientFlowsMetric2008}.

\begin{proof}[proof of Theorem~\ref{thm:regularizing_effects}]
    Let $\tilde{u}$ be the arc-length parametrization of $u$ and $\sa_u, \ta{u}$ be as in Lemma~\ref{lem:arc_length_param}.
    
    First, let us check (i).
    If $\ta{u} < a$, then the locally integrability of $|u'|$ implies that $S^\ast_u < +\infty$.
    Thus we can extend the result of Lemma~\ref{lem:convexity_of_f_circ_u_tilde} on any interval $I \subset [0, S^\ast_u]$ with $|I| \leq 1$, which induces that $\localslope \circ \tilde{u}$ is locally bounded from above on $[0, S^\ast_u]$.
    Therefore we can assume that $\ta{u} = a$.
    For any $[T_0, T_1] \subset (0, a)$, the local upper bound of $\localslope$ yields the existence of $M > 0$ such that for any $t_0, t_1 \in [T_0, T_1]$ with $t_0 \leq t_1$
    \begin{align*}
        d(u(t_0), u(t_1))
        \leq \int_{t_0}^{t_1}|u'|(r) dr 
        \underset{\eqref{eq:energy_identity}}{\leq}
        M \cdot (t_1 - t_0),
    \end{align*}
    and
    \begin{align*}
        |f(u(t_1)) - f(u(t_0))| 
        \leq \int_{t_0}^{t_1} |(f\circ u)'(r)|dr 
        \underset{\eqref{eq:energy_identity}}{\leq}
        M \cdot (t_1 - t_0).
    \end{align*}
    These show the local Lipschitz continuity of $u$ and $f \circ u$.
    Let us take $A \coloneqq \{ s \in [0, S^\ast_u) \mid (f \circ \tu)'(s) = - \localslope \circ \tu (s) \}$ that has full $\ms{L}$-measure in $[0, \sa_u)$.
    Since $f \circ \tu$ is $(2, -\lambda^-)$-convex on every $I \subset [0, \sa_u)$ with $|I| \leq 1$, we have for any $s, t \in A$ with $0 \leq s - t < 1$
    \begin{align}
        \label{eq:slope_noninc_lambda_convex}
        \localslope \circ \tu(s) \leq \localslope \circ \tu(t) + 2\lambda^-|t - s|.
    \end{align}
    Fixing any $s \in (0, \sa_u)$ and then taking $s_0 \in A$ with $0 < s - s_0 < 1$, we get
    \begin{align*}
        \localslope \circ \tu(s)
        \underset{\text{l.s.c.}}{\leq} \liminf_{A \ni s' \uparrow s}\localslope \circ \tu(s')
        \underset{\eqref{eq:slope_noninc_lambda_convex}}{\leq} \localslope \circ \tu(s_0) + 2\lambda^-|s - s_0| < + \infty.
    \end{align*}
    These complete our discussion of (i).
    
    Next, let us check (ii).
    For any $s \in [0, \sa_u)$ and $s' \in A$ with $s < s'$ letting $h \coloneqq s' - s$ we have
    \begin{align*}
        &h \localslope \circ \tu(s')
        \underset{\eqref{eq:slope_noninc_lambda_convex}}{\leq} \int_{s}^{s'} \{ \localslope \circ \tu(r) + 2 \lambda^-|s' -r|\}dr\\
        &\underset{\scriptsize \text{Prop.}~\ref{prop:property_of_arc_length_param}}{=} f \circ \tu(s) - f \circ \tu(s') + \lambda^-h^2
        \underset{\scriptsize \text{Prop.}~\ref{prop:local_slope_is_a_strong_upper_gradient}}{\leq} \localslope \circ \tu(s)h + \lambda^-(h^{p_0} + h^2);
    \end{align*}
    hence dividing both sides by $h > 0$ and then letting $A \ni s'\:\downarrow\: s$ we get
    \begin{align*}
        \limsup_{A \ni s' \downarrow s} \localslope \circ \tu (s') \leq \localslope \circ \tu(s) \quad \text{for every } s \in [0, \sa_u);
    \end{align*}
    therefore the lower semicontinuity of $\localslope \circ \tu$ implies the almost right continuity of $\localslope \circ \tu$.
    Writing $\params^{-1}(A)$ again as $A$, since $\paramt$ is locally Lipschitz on $(0, \sa_u)$, we get the almost right continuity of $\localslope \circ u$ on $[0, a)$ w.r.t. $A$.
    This yields that
    \begin{align}
        \label{eq:right_derivative_1}
        (f\circ u)'_+(t) \underset{\eqref{eq:energy_identity}}{=} \lim_{h\:\downarrow\:0}\frac{1}{h} \int_{t}^{t+h} -\localslope^q \circ u(r) dr
        =  -\localslope^q \circ u(t) 
    \end{align}
    for every $t \in (0, a)$; furthermore this holds at $t=0$ if $u_0 \in D(\localslope)$.
    Similarly, we obtain
    \begin{align*}
        \limsup_{h\:\downarrow\:0} \frac{d(u(t+h), u(t))}{h}
        \underset{\eqref{eq:energy_identity}}{\leq} \limsup_{h\:\downarrow\:0} \frac{1}{h} \int_{t}^{t+h} \localslope^{q/p}\circ u(r)dr
        = \localslope^{q/p}\circ u(t)
    \end{align*}
    for every $t \in (0, a)$ and for $t=0$ if $u_0 \in D(\localslope)$.
    The global formula~\eqref{eq:global_formula} of $\localslope$ induces that for any $t \in [0, a)$ and $h > 0$
    \begin{align*}
        \frac{d(u(t+h), u(t))}{h}\localslope\circ u(t)
        \geq \frac{f(u(t)) - f(u(t+h))}{h} - \lambda^-\frac{d^{p_0}(u(t+h), u(t))}{h};
    \end{align*}
    hence \eqref{eq:right_derivative_1} yields
    \begin{align*}
        \liminf_{h\:\downarrow\:0}\frac{d(u(t+h), u(t))}{h}
        \geq \localslope^{q/p} \circ u(t)
    \end{align*}
    for every $t \in (0, a)$ and for $t = 0$ if $u_0 \in D(\localslope)$.
    This ends the check of~\eqref{eq:regularizing_right_energy_identity} and \eqref{eq:regularizing_right_approximate_cont}.\\
    If $\lambda \geq 0$, then we can extend the convexity of $f\circ \tu$ over $[0, \sa_u)$;
    therefore there exists $A \subset [0, \sa_u)$ with full $\ms{L}$-measure in $[0, \sa_u)$ such that $\localslope \circ \tu$ is non-increasing on $A$.
    For any $s \in (0, \sa_u)$ and $t_0 \in A$ with $t_0 < s$
    \begin{align*}
        \localslope \circ \tu(s)
        \underset{\text{l.s.c.}}{\leq} \liminf_{A \ni s' \:\uparrow\:s} \localslope\circ\tu(s') 
        \leq \localslope\circ\tu(t_0);
    \end{align*}
    on the other hand, for any $s \in [0, \sa_u)$ and $t_0 \in A$ with $s < t_0$, letting $h \coloneqq t_0 - s$, the global formula of $\localslope$ gives
    \begin{align*}
        \localslope\circ\tu(t_0)
        \leq \frac{1}{h}\int_{s}^{t_0} \localslope \circ \tu(r)dr
        \underset{\scriptsize \text{Prop.}~\ref{prop:property_of_arc_length_param}}{\leq}
        \frac{f(\tu(s)) - f(\tu(t_0))}{h}
        \leq \localslope\circ\tu(s).
    \end{align*}
    Therefore $\localslope \circ \tu$ is non-increasing on $[0, \sa_u)$, leading to that $\localslope \circ u$ is non-increasing on $[0, a)$.
    Because $\localslope \circ u$ is lower semicontinuous and non-increasing, it is right continuous.
    This completes the proof of (ii).

    Note that (iv) is a direct consequence of Proposition~\ref{prop:minimizer_lambda_gt_0}, thus let us finally check (iii).
    Since $\localslopeu$ is non-increasing, we have for any $t, t_0 \in [0, a)$ with $t_0 < t$
    \begin{align*}
        (t - t_0) \localslope^q \circ u(t)
        \leq \int_{t_0}^t \localslope^q \circ u(r)dr
        \underset{\eqref{eq:energy_identity}}{=} f(u(t_0)) - f(u(t))
    \end{align*}
    and
    \begin{align*}
        &\frac{(t - t_0)}{q} \localslope^q \circ u(t)
        \leq \frac{1}{q} \int_{t_0}^t \localslope^q \circ u(r)dr\\
        &\underset{\eqref{eq:energy_identity}}{=} f(u(t_0)) - f(u(t)) - \frac{1}{p} \int_{t_0}^t |u'|^p(r)dr\\
        &\leq f(u(t_0)) - f(u(t)) - \frac{1}{p(t-t_0)^{p-1}}d^p(u(t_0), u(t)).
    \end{align*}
    This completes our check of (iii).
\end{proof}

\section{Existence of canonical bijective maps and uniqueness of \texorpdfstring{$p$}{p}-curves of maximal slope}
\label{sec:bijective_map_and_uniqueness}
Let $(\ms{S},d)$ be a complete metric space, $f \colon \ms{S} \to (-\infty, +\infty]$ be a proper and lower semicontinuous, and $u_0 \in D(f)$.
This section aims to prove the existence of canonical bijective maps from $p$-curves of maximal slope to $p'$-ones, and the uniqueness results -Corollaries~\ref{cor:uniqueness_of_p_cms_lambda_gt_0_intro} and \ref{cor:uniqueness_of_p_cms_lambda_st_0_intro}-.

For any $u_0 \in D(f)$, let
\begin{align*}
    \text{CMS}_p^\ast(u_0) \coloneqq \bigcup_{a \in (0, +\infty]} \{ u \in \pcms{p}{a} \mid \ta{u} = a\},
\end{align*}
and
\begin{align*}
    \text{CMS}_{p, \infty}(u_0) \coloneqq \{ u \in \pcms{p}{+\infty} \mid \exists \lim_{t\to+\infty}u(t) \}.
\end{align*}
Recalling Proposition~\ref{prop:positivity_of_localslope}, for any $a \in (0, +\infty]$ we can consider
\begin{align*}
 \pcms{p}{a} \setminus \text{CMS}_p^\ast(u_0) \subset \text{CMS}_{p, \infty}(u_0),
\end{align*}
with appropriate extensions of the curves' domain if necessary.

Let us construct canonical bijective maps that make clear the difference between $p$ and $p'$-curves of maximal slope.
\begin{cor}[Canonical bijective maps between $p$-curves of maximal slope and $p'$-ones]
    \label{cor:existence_of_bijective_map}
    Let one of the following hold:
    \begin{enumerate}[label=(\alph*)]
        \item $\lambda=0$ and $p_0 \in (1, +\infty)$;
        \item $\lambda < 0$ and $p_0 \in [2, +\infty)$.
    \end{enumerate}
    Suppose that $f$ is $(p_0, \lambda)$-convex.
    For any $p, p' \in (1, + \infty)$ if $\lambda = 0$ or any $p, p' \in (1, p_0]$ if $\lambda < 0$, we define the map $\Lambda_1 \colon \text{CMS}_p^\ast(u_0) \to \text{CMS}_{p'}^\ast(u_0)$ and the map $\Lambda_2 \colon \text{CMS}_{p, \infty}(u_0) \to \text{CMS}_{p', \infty}(u_0)$ by
    \begin{align*}
        \Lambda_1(u) &\coloneqq u \circ \tilde{\paramt}_{p \to p'} \quad \text{for any } u \in \text{CMS}_p^\ast(u_0), \\
        \Lambda_2(u) &\coloneqq u \circ \paramt_{p \to p'} \quad \text{for any } u \in \text{CMS}_{p, \infty}(u_0),
    \end{align*}
     where $\tilde{\paramt}_{p \to p'}$ is as in Lemma~\ref{lem:param_trans_domain_possibly_bounded}, and $\paramt_{p\to p'}$ is as in Theorems~\ref{thm:p_cms_to_p_dash_cms_lambda_gt_0}, \ref{thm:p_cms_to_p_dash_cms_lambda_st_0} if $\lambda = 0$, if $\lambda < 0$, respectively.
     Then they are bijective.
\end{cor}

Next, focusing on $p$-curves of maximal slope defined over the interval $[0, \infty)$, let's prove Corollary~\ref{cor:uniqueness_of_p_cms_lambda_gt_0_intro}.
\begin{cor}[Uniqueness of $p$-curves of maximal slope; the case $\lambda = 0$]
    \label{cor:uniqueness_p_cms_lambda_gt_0}
    Suppose that $f$ is $(p_0, \lambda)$-convex for some $(p_0, \lambda)$ satisfying $\lambda = 0$ and $p_0 \in (1, +\infty)$.
    For any $p, p' \in (1, + \infty)$  we define the map $\Lambda \colon \pcms{p}{+\infty} \to \pcms{p'}{+\infty}$ by
    \begin{align*}
        \Lambda(u) \coloneqq u \circ \paramt_{p \to p'} \quad \text{for any } u \in \pcms{p}{+\infty},
    \end{align*}
    where $\paramt_{p \to p'}$ is as in Theorem~\ref{thm:p_cms_to_p_dash_cms_lambda_gt_0}.
    Then this is bijective.

    In particular, if $\# \pcms{p}{+\infty} = 1$ for an exponent $p \in (1, +\infty)$, then $\# \pcms{p'}{+\infty} = 1$ for all $p' \in (1, +\infty)$.
\end{cor}

By the above corollary, we have done the proof of Corollary~\ref{cor:uniqueness_of_p_cms_lambda_gt_0_intro}.

Finally, let us prove Corollary~\ref{cor:uniqueness_of_p_cms_lambda_st_0_intro}.
\begin{cor}[Uniqueness of $p$-curves of maximal slope; the case $\lambda < 0$]
    \label{cor:uniqueness_p_cms_lambda_st_0}
    Suppose that $f$ is $(p_0, \lambda)$-convex for some $(p_0, \lambda)$ satisfying $\lambda < 0$ and $p_0 \in [2, +\infty)$.
    If there exists an exponent $p \in (1, p_0]$ such that
    \begin{align}
        \label{ass:uniqueness_p_cms_for_all_a}
        \#\pcms{p}{a}=1 \quad \text{for any } a \in (0, +\infty],
    \end{align}
    then for any $p' \in (1, p_0]$ we have
    \begin{align*}
        &T_{p'} \coloneqq \sup \{a \in (0, +\infty] \mid \# \pcms{p'}{a} \geq 1\}\\ 
        & = 
        \begin{cases}
            \sa_{p\to p'} & \text{(i) } \sa_{p \to p'} < +\infty, \ta{u} = +\infty \text{ and } \nexists \lim_{t\:\uparrow\:\ta{u}}u(t) \\
            +\infty & \text{(ii) otherwise}
         \end{cases},
    \end{align*}
    and the map $\Lambda \colon \pcms{p}{+\infty} \to \pcms{p'}{T_{p'}}$ defined to be
    \begin{align*}
        \Lambda(u) \coloneqq \begin{cases}
            u \circ \tilde{\paramt}_{p \to p'} & \text{if (i)}\\
            u \circ \paramt_{p \to p'} & \text{if (ii)}
        \end{cases}
    \end{align*}
    is bijective, where $u$ is the unique element in $\pcms{p}{+\infty}$, and $\tilde{\paramt}_{p\to p'}$, $\paramt_{p\to p'}$ is as in Lemma~\ref{lem:param_trans_domain_possibly_bounded}, Theorem~\ref{thm:p_cms_to_p_dash_cms_lambda_st_0}, respectively.
\end{cor}

\begin{proof}
    Let $u$ be the unique element of $\pcms{p}{+\infty}$.

    First, let us consider the case (ii).
    By Theorem~\ref{thm:p_cms_to_p_dash_cms_lambda_st_0}, we have $T_{p'} = +\infty$.
    The injectivity of $\Lambda$ is clear; thus we check the subjectivity of it.
    Let $v_{p'} \in \pcms{p'}{+\infty}$.
    To use Theorem~\ref{thm:p_cms_to_p_dash_cms_lambda_st_0} let us check that the condition (c) in Theorem~\ref{thm:p_cms_to_p_dash_cms_lambda_st_0} holds.
    Suppose that $\sa_{p' \to p} < +\infty$ and $\ta{v_{p'}} = +\infty$ for $v_{p'}$.
    Letting $v \coloneqq v_{p'} \circ \tilde{\paramt}_{p' \to p}$ where $\tilde{\paramt}_{p'\to p}$ is as in Lemma~\ref{lem:param_trans_domain_possibly_bounded} for $v_{p'}$, the uniqueness assumption~\eqref{ass:uniqueness_p_cms_for_all_a} induces that $u = v$ on $[0, \sa_{p' \to p})$; hence we get
    \begin{align*}
        \lim_{t \:\uparrow\: \ta{v_{p'}}} v_{p'}(t)
        = \lim_{t \:\uparrow\: \ta{v_{p'}}} v \circ \tilde{\params}_{p' \to p}(t)
        = \lim_{s \:\uparrow\: \sa_{p' \to p}} u(s) = u(\sa_{p' \to p}).
    \end{align*}
    This allow us to use Theorem~\ref{thm:p_cms_to_p_dash_cms_lambda_st_0}, leading to the equality
    \begin{align*}
        v_{p'} = \Lambda(v_{p'} \circ \paramt_{p' \to p}).
    \end{align*}
    This completes our discussion in case (ii).

    Next, let us deal with the remaining case (i).
    By Lemma~\ref{lem:param_trans_domain_possibly_bounded}, we have $T_{p'} \geq \sa_{p \to p'}$.
    We check that $T_{p'} = \sa_{p \to p'}$.
    Let $v \in \pcms{p'}{a}$.
    If $\ta{v} < a$, then letting $v_p \coloneqq v \circ \tilde{\paramt}_{p' \to p}$ the uniqueness assumption~\eqref{ass:uniqueness_p_cms_for_all_a} gives
    \begin{align*}
        \lim_{s\:\uparrow\: \sa_{p' \to p}}u(s)
        = \lim_{s\:\uparrow\: \sa_{p' \to p}} v_p(s)
        = v(\ta{v}).
    \end{align*}
    This is a contradiction to $\nexists \lim_{t\:\uparrow\:\ta{u}}u(t)$; thus $\ta{v} = a$.
    Corollary~\ref{cor:existence_of_bijective_map} and the assumption~\eqref{ass:uniqueness_p_cms_for_all_a} show the existence of $b \in (0, +\infty]$ such that $\Lambda_1(u|_{[0, b)}) = v$, which leads to $a = \sa_{p\to p', u|_{[0, b)}} \leq \sa_{p\to p'}$.
    This completes our check of $T_{p'} = \sa_{p \to p'}$.
    Then, the bijectivity of $\Lambda$ is proved as follows.
    The injectivity of $\Lambda$ is clear.
    Let $v_{p'} \in \pcms{p'}{T_{p'}}$.
    Similarly as the above discussion, we have $\ta{v_{p'}} = T_{p'}$.
    If $\sa_{p' \to p}$ for $v_{p'}$ satisfies  $\sa_{p' \to p} < \ta{u}$, then letting $v \coloneqq v_{p'} \circ \tilde{\paramt}_{p' \to p}$  
    this gives
    \begin{align*}
        \sa_{p \to p'}
        = \ta{v_{p'}}
        = \tilde{\paramt}_{p' \to p, v_{p'}}(\sa_{p'\to p})
        = \tilde{\params}_{p \to p', v}(\sa_{p'\to p})
        \underset{\eqref{ass:uniqueness_p_cms_for_all_a}}{=} \tilde{\params}_{p \to p', u}(\sa_{p'\to p})
        < \sa_{p \to p'}.
    \end{align*}
    This is a contradiction; thus $\sa_{p' \to p} = \ta{u} = +\infty$. 
    This implies that $v$ is a member of $\pcms{p}{+\infty}$;
    therefore $v_{p'} = \Lambda(v)$.
\end{proof}

The condition $p' \in (1, p_0]$ is essential in Corollary~\ref{cor:uniqueness_p_cms_lambda_st_0}:
\begin{ex}
    \label{ex:p_cms_are_not_unique_for_p_gt_p_0}
    We consider a $2$-dimentional Euclidian space $(\mb{R}^2, d)$ as $(\ms{S}, d)$.
    Let $f(x) \coloneqq - \frac{1}{2}(x_1^2 + x_2^2)$.
    This function $f$ is $(2, -2)$-convex and Assumption~\eqref{ass:uniqueness_p_cms_for_all_a} holds for $p=2$ because of the following: 
    any $(2, \lambda)$-convex function with some $\lambda \in \mb{R}$ on a Hilbert space generates an EVI-curve starting from $u_0$ for each $u_0 \in D(f)$ (see Section~\ref{subsec:applications} for more details).

    Let $u_0 = (0, 0)$. This is a critical point of $f$ in the sense $\localslope (u_0) = 0$ so that the unique element $u \in \pcms{p'}{\sap{2}{p'}}$ satisfies $u(t) = u_0$ for any $t \in [0, \sap{2}{p'})$ for all $p' \in (1, 2]$.
    On the other hand for any $p' \in (2, +\infty)$ letting $\alpha \coloneqq 1 - \frac{q'}{p'} \in (0, 1)$ the curve $u(t) \coloneqq \left( \cos \theta \cdot (\alpha t)^{1/\alpha}, \sin \theta \cdot (\alpha t)^{1/\alpha} \right)$ belongs to $\pcms{p'}{+\infty}$ for any $\theta \in [0, 2\pi)$; therefore $p'$-curves of maximal slope are not unique.
\end{ex}

\subsection{Application: unique existence of \texorpdfstring{$p$}{p}-curves of maximal slope when EVI-flow exists}
\label{subsec:applications}

In this section, we aim to show the unique existence of $p$-curves of maximal slope under the existence assumption of the evolution variational inequality (EVI) gradient flow which is introduced in~\cite{ambrosioGradientFlowsMetric2008} and studied in~\cite{muratori2020108347}.
In the sequel, let $\lambda \in \mb{R}$, and we describe $(2, \lambda)$-convex as $\lambda$-convex for simplicity.

An $\evilam$-curve $u \colon [0, +\infty) \to \ms{S}$ is a continuous map satisfying $u(t) \in D(f)$ for every $t > 0$ and the \textit{evolution variational inequality}
\begin{align*}
    \frac{d}{dt}^+ d^2(u(t), V) + \frac{\lambda}{2}d^2(u(t), V) \leq f(V) - f(u(t)) \text{ for every } t > 0, V \in D(f),
\end{align*}
where $\frac{d}{dt}^+\phi(t)$ is the right upper Dini derivative $\limsup_{h\downarrow 0} \frac{\phi(t+h) - \phi(t)}{h}$ at $t \in (0, +\infty)$.
$\evilam$-curves induce useful properties; we quote a few of them from~\cite[Theorem~3.5]{muratori2020108347}:

\begin{thm}[\cite{muratori2020108347}]
    Let $u, u^1, u^2$ be $\evilam$-curves.
    Then we have
    \begin{itemize}
        \item (\textbf{$\lambda$-contraction}) for every $0 \leq s \leq t < +\infty$
        \begin{align*}
            d(u^1(t), u^2(t)) \leq e^{-\lambda(t-s)}d(u^1(s), u^2(s)).
        \end{align*}
        In particular, for each $v_0 \in \cldomain$ there is at most one $\evilam$-curve $u$ starting from $v_0$.
        \item (\textbf{Energy identity}) for every $t > 0$ the right limits
            \begin{align*}
                |u'_+|(t) \coloneqq \lim_{h \downarrow 0} \frac{d(u(t+h), u(t))}{h},
                (f\circ u)'_+(t) \coloneqq \lim_{h \downarrow 0} \frac{f\circ u(t + h) - f\circ u(t)}{h}
            \end{align*}
            exist, satisfy
            \begin{align*}
                (f \circ u)'_+(t) = - |u'_+|^2(t) = - \localslope^2\circ u(t),
            \end{align*}
            and are right continuous on $(0, +\infty)$.
    
            Moreover the map $t \mapsto f \circ u(t)$ is non-increasing continuous on $[0, +\infty)$.
    \end{itemize}
\end{thm}
Note that, from the energy identity in the above theorem, for any $v_0 \in D(f)$ an $\evilam$-curve starting from $v_0$ is a $2$-curves of maximal slope starting from $v_0$ defined on $[0, +\infty)$.

An $\evilam$-flow is a family of continuous maps $\{S_t \colon D(f) \to D(f)\}_{t\geq0}$ satisfying that, for any $v_0 \in D(f)$, 
$S_0[v_0] = v_0$
and the map $t \in [0, +\infty) \mapsto u(t) \coloneqq S_t[v_0]$ is an $\evilam$-curve.
Notably, it has been shown that
\begin{align}
    \label{cond:evi_lambda_exists}
    \text{an }\evilam\text{-flow } \{S_t\}_{t\geq0} \text{ exists} \tag{EVI$_\lambda$}
\end{align}
for several interesting settings listed below.
\begin{itemize}[nolistsep, leftmargin=*]
    \item any $\lambda$-convex function along geodesics on connected and complete smooth Riemannian manifolds (see e.g.~\cite{villani2008optimal}), where a geodesic is a curve $\gamma $ from $[0, 1]$ to $\ms{S}$ such that $d(\gamma_t, \gamma_s) = d(\gamma_0, \gamma_1)|s-t|$ for every $s , t \in [0, 1]$ with $s \leq t$;
    \item any $\lambda$-convex function along geodesics on complete CAT$(0)$ spaces in~\cite{Mayer1998GradientFO}, moreover on complete CAT$(k)$-spaces with $k \in \mb{R}$ in~\cite{ohtapalfia2017};
    \item any $\lambda$-convex function along \textit{generalized geodesics} (see~\cite[Definition~9.2.4]{ambrosioGradientFlowsMetric2008}) on 2-Wasserstein spaces over separable Hilbert spaces in~\cite{ambrosioGradientFlowsMetric2008};
    \item the relative entropy functional on 2-Wasserstein spaces over connected and complete Riemannian manifolds in~\cite{erbar2010};
    \item any continuous $\lambda$-convex functions along geodesics, satisfying an appropriate lower bound condition, on locally compact complete RCD$(K, +\infty)$ spaces for some $K\in \mb{R}$ in~\cite{strum2018gradflow}.
\end{itemize}
We can also find a result in~\cite{savare07diffusionsemi} (without a proof), stating that \eqref{cond:evi_lambda_exists} holds for any $\lambda$-convex function along geodesics on positively curved Alexandrov spaces or 2-Wasserstein spaces over them.
See~\cite[Introduction and Section~3.4]{muratori2020108347} for more examples.

We recall the following unique existence result shown by Muratori and Savar\'e in~\cite[Theorem~4.2]{muratori2020108347}:
\begin{thm}[\cite{muratori2020108347}]
    \label{thm:muratori_uniqueness_2_cms}
    Suppose that $(\ms{S}, d)$ and $f$ satisfy \eqref{cond:evi_lambda_exists} for some $\lambda \in \mb{R}$.
    Then for any $a \in (0, +\infty]$ there exists the unique $2$-curve of maximal slope starting from $u_0$ defined on $[0, a)$.
\end{thm}

Combining Theorem~\ref{thm:muratori_uniqueness_2_cms} with our uniqueness results -Corollaries~\ref{cor:uniqueness_p_cms_lambda_gt_0} and \ref{cor:uniqueness_p_cms_lambda_st_0}-, we get the unique existence of $p$-curves of maximal slope even if $p\neq2$:
\begin{thm}[Unique existence of $p$-curves of maximal slope; the case $\lambda \geq 0$]
    \label{thm:application_unique_existence_p_cms_lambda_gt_0}
    Suppose that $f$ is $(p_0, \lambda_0)$-convex for some $(p_0, \lambda_0)$ satisfying $\lambda_0 = 0$ and $p_0 \in (1, +\infty)$, and $(\ms{S}, d)$ and $f$ satisfy \eqref{cond:evi_lambda_exists} for some $\lambda \in \mb{R}$.
    Then for all $p \in (1, +\infty)$ there exists the unique $p$-curve of maximal slope starting from $u_0$ defined on $[0, +\infty)$.
\end{thm}

\begin{thm}[Unique existence of $p$-curves of maximal slope; the case $\lambda < 0$]
    \label{thm:application_unique_existence_p_cms_lambda_st_0}
    Suppose that $f$ is $(p_0, \lambda_0)$-convex for some $(p_0, \lambda_0)$ satisfying $\lambda_0 < 0$ and $p_0 \in [2, +\infty)$, and $(\ms{S}, d)$ and $f$ satisfy \eqref{cond:evi_lambda_exists} for some $\lambda \in \mb{R}$.
    Then for all $p \in (1, p_0]$ there exists the unique $p$-curves of maximal slope starting from $u_0$ defined on $[0, T_p)$ and
    \begin{align*}
        T_p = \sup\{a \in (0, +\infty] \mid \# \pcms{p}{a} \geq 1\},
    \end{align*}
    where $T_p$ is as in Corollary~\ref{cor:uniqueness_p_cms_lambda_st_0}.
\end{thm}

\begingroup
\setlength\bibitemsep{0pt}
\printbibliography

@article{muratori2020108347,
    title = {Gradient flows and Evolution Variational Inequalities in metric spaces. {I}: Structural properties},
    journal = {Journal of Functional Analysis},
    volume = {278},
    number = {4},
    year = {2020},
    author = {Matteo Muratori and Giuseppe Savar{\'e}},
}

@book{ambrosioGradientFlowsMetric2008,
  title = {Gradient Flows in Metric Spaces and in the Space of Probability Measures},
  edition = {2nd},
  series = {Lectures in Mathematics {{ETH Z\"urich}}},
  publisher = {{Birkh\"auser}},
  year = {2008},
  author = {Ambrosio, Luigi and Gigli, Nicola and Savar{\'e}, Giuseppe},
}

@article{ohta2023gradient,
      title={Gradient flows in asymmetric metric spaces and applications}, 
      author={{Shin-ichi} Ohta and Wei Zhao},
      year={2023},
      eprint={2206.07591},
      archivePrefix={arXiv}
}

@book{villani2008optimal,
title={Optimal Transport: Old and New},
author={Villani, C{\'e}dric},
series={Grundlehren der mathematischen Wissenschaften},
year={2008},
publisher={Springer Berlin Heidelberg}
}

@misc{caillet2024doubly,
    title={Doubly nonlinear diffusive PDEs: new existence results via generalized Wasserstein gradient flows},
    author={Thibault Caillet and Filippo Santambrogio},
    year={2024},
    eprint={2402.02882},
    archivePrefix={arXiv},
    primaryClass={math.AP}
}

@article{Agueh2005,
    author = {Martial Agueh},
    title = {{Existence of solutions to degenerate parabolic equations via the Monge-Kantorovich theory}},
    volume = {10},
    journal = {Advances in Differential Equations},
    number = {3},
    publisher = {Khayyam Publishing, Inc.},
    pages = {309 -- 360},
    year = {2005},
}

@inproceedings{Otto1996DoublyDD,
    title={Doubly Degenerate Diffusion Equations as Steepest Descent},
    author={Felix Otto},
    year={1996},
    note={preprint}
}

@Article{
    rosmiesav08,
    author = { Mielke, Alexander and Rossi, Riccarda and Savar{\'e}, Giuseppe },
    title = { A metric approach to a class of doubly nonlinear evolution equations and applications },
    journal = { Ann. Sc. Norm. Super. Pisa Cl. Sci. (5) },
    year = { 2008 },
    volume = { 7 },
    pages = { 97-169 },
}

@article{colli1992,
    title = {On some doubly nonlinear evolution equations in Banach spaces},
    journal = {Japan Journal of Industrial and Applied Mathematics},
    volume = {9},
    year = {1992},
    author = {Colli, Pierluigi},
}

@article{Hynd2016ADN,
    title={A doubly nonlinear evolution for the optimal Poincar{\'e} inequality},
    author={Ryan Hynd and Erik Anders Lindgren},
    journal={Calculus of Variations and Partial Differential Equations},
    year={2016},
    volume={55},
}

@article{HYND20174873,
    title = {Approximation of the least Rayleigh quotient for degree p homogeneous functionals},
    journal = {Journal of Functional Analysis},
    volume = {272},
    number = {12},
    pages = {4873-4918},
    year = {2017},
    author = {Ryan Hynd and Erik Lindgren},
}

@article{Mayer1998GradientFO,
    title={Gradient flows on nonpositively curved metric spaces and harmonic maps},
    author={Uwe F. Mayer},
    journal={Communications in Analysis and Geometry},
    year={1998},
    volume={6},
    pages={199-253},
}

@article{ohtapalfia2017,
    title = {Gradient flows and a Trotter-Kato formula of semi-convex functions on CAT(1)-spaces},
    author = {Shinichi Ohta and Mikl\'os P\'alfia},
    journal = {American Journal of Mathematics},
    number = {4},
    pages = {937--965},
    publisher = {Project Muse},
    volume = {139},
    year = {2017}
}

@article{savare07diffusionsemi,
    author = {Savar{\'e}, Giuseppe},
    title = {Gradient flows and diffusion semigroups in metric spaces under lower curvature bounds},
    journal = {Comptes Rendus. Math{\'e}matique},
    pages = {151--154},
    publisher = {Elsevier},
    volume = {345},
    number = {3},
    year = {2007},
}

@article{erbar2010,
    author = {Matthias Erbar},
    title = {{The heat equation on manifolds as a gradient flow in the Wasserstein space}},
    volume = {46},
    journal = {Annales de l'Institut Henri Poincar{'e}, Probabilités et Statistiques},
    number = {1},
    publisher = {Institut Henri Poincar{\'e}},
    pages = {1 -- 23},
    year = {2010},
}

@article{strum2018gradflow,
    author = {Sturm, Karl-Theodor},
    year = {2018},
    month = {9},
    title = {Gradient Flows for Semiconvex Functions on Metric Measure Spaces - Existence, Uniqueness and Lipschitz Continuity},
    volume = {146},
    pages={3985–-3994},
    journal = {Proceedings of the American Mathematical Society},
}

@article{Ambrosio_2014,
   title={Metric measure spaces with Riemannian Ricci curvature bounded from below},
   volume={163},
   number={7},
   journal={Duke Mathematical Journal},
   publisher={Duke University Press},
   author={Ambrosio, Luigi and Gigli, Nicola and Savaré, Giuseppe},
   year={2014}
}

@book{burago2001course,
  title={A Course in Metric Geometry},
  author={Burago, D. and Burago, I.U.D. and Ivanov, S.},
  series={Crm Proceedings \& Lecture Notes},
  year={2001},
  publisher={American Mathematical Society}
}

@article{KELL20162045,
    title = {{q}-Heat flow and the gradient flow of the Renyi entropy in the p-Wasserstein space},
    journal = {Journal of Functional Analysis},
    volume = {271},
    number = {8},
    pages = {2045-2089},
    year = {2016},
    author = {Martin Kell},
}

@article{ROSSI2011205,
    title = {Global attractors for gradient flows in metric spaces},
    journal = {Journal de Mathématiques Pures et Appliquées},
    volume = {95},
    number = {2},
    pages = {205-244},
    year = {2011},
    author = {Riccarda Rossi and Antonio Segatti and Ulisse Stefanelli},
}

@article{XU19911127,
    title = {Inequalities in Banach spaces with applications},
    journal = {Nonlinear Analysis: Theory, Methods and Applications},
    volume = {16},
    number = {12},
    pages = {1127-1138},
    year = {1991},
    author = {Hong-Kun Xu},
}

@article{Naor_Silberman_2011,
    title={Poincaré inequalities, embeddings, and wild groups},
    volume={147},
    number={5}, 
    journal={Compositio Mathematica}, 
    author={Naor, Assaf and Silberman, Lior}, 
    year={2011}, 
    pages={1546–1572}
}
\endgroup
{\small
 (Sho Shimoyama), \textsc{Graduate School of Mathematical Sciences, The University of Tokyo, Tokyo, Japan}\\
\textit{Email address}: \href{mailto:sho-shimoyama@g.ecc.u-tokyo.ac.jp}{sho-shimoyama@g.ecc.u-tokyo.ac.jp}
}

\end{document}